\documentclass[11pt]{amsart}
\usepackage[margin=1.1in]{geometry}
\usepackage{setspace}
\usepackage{graphicx}
\usepackage{enumerate}
\usepackage{amssymb}
\usepackage{tikz-cd}
\usepackage{mathrsfs}
\usepackage{extarrows}
\usepackage{url}
\numberwithin{equation}{subsection}
\newtheorem{theorem}[subsubsection]{Theorem}
\newtheorem{lemma}[subsubsection]{Lemma}

\newtheorem{proposition}[subsubsection]{Proposition}
\newtheorem{corollary}[subsubsection]{Corollary}

\theoremstyle{definition}
\newtheorem{definition}[subsubsection]{Definition}

\theoremstyle{remark}
\newtheorem{remark}[subsubsection]{Remark}

\let\realequation\equation
\def\equation{\setcounter{equation}{\arabic{subsubsection}}%
   \refstepcounter{subsubsection}%
   \realequation}

\newcommand{\Isog}{\mathrm{Isog}}

\newcommand{\ad}{\mathrm{ad}}

\newcommand{\Gr}{\mathrm{Gr}}
\newcommand{\Sh}{\mathrm{Sh}}
\newcommand{\ord}{\mathrm{ord}}
\newcommand{\Frob}{\mathrm{Frob}}
\newcommand{\Lie}{\mathrm{Lie}}

\newcommand{\sS}{\mathscr{S}}
\newcommand{\bbA}{\mathbb{A}}
\newcommand{\bbC}{\mathbb{C}}
\newcommand{\bbD}{\mathbb{D}}
\newcommand{\bbF}{\mathbb{F}}
\newcommand{\bbG}{\mathbb{G}}

\newcommand{\bbR}{\mathbb{R}}
\newcommand{\bbQ}{\mathbb{Q}}
\newcommand{\bbZ}{\mathbb{Z}}

\newcommand{\calH}{\mathcal{H}}
\newcommand{\calL}{\mathcal{L}}

\newcommand{\calS}{\mathcal{S}}

\newcommand{\frakp}{\mathfrak{p}}

\newcommand{\frakg}{\mathfrak{g}}
\newcommand{\frakh}{\mathfrak{h}}

\newcommand{\GSp}{\mathrm{GSp}}
\newcommand{\GSpin}{\mathrm{GSpin}}
\newcommand{\GL}{\mathrm{GL}}
\newcommand{\der}{\mathrm{der}}
\newcommand{\an}{\mathrm{an}}
\newcommand{\RZ}{\mathrm{RZ}}

\newcommand{\Qell}{\bar{\mathrm{Q}}_\ell}

\newcommand{\barSs}{\bar{\mathscr{S}}}
\newcommand{\Cor}{\mathrm{Cor}}
\newcommand{\IC}{\mathrm{IC}}

\usepackage{hyperref}
\begin{document}

\title{Semisimplicity of \'{e}tale cohomology of certain Shimura varieties}

\author{Si Ying Lee}

\begin{abstract}
Building on work of Fayad and Nekov\'{a}\v{r}, we show that a certain part of the etale cohomology of some abelian-type Shimura varieties is semisimple, assuming the associated automorphic Galois representations exists, and satisfies some good properties. The proof combines an abstract semisimplicity criterion of Fayad-Nekov\'{a}\v{r} with generalized Eichler-Shimura relations for abelian type Shimura varieties and partial Frobenii. We apply this to some symplectic and orthogonal Shimura varieties.
\end{abstract}

\maketitle
\setcounter{tocdepth}{1}
\section{Introduction}
Let $G$ be a reductive group over $\bbQ$, and $X$ be a conjugacy class of homomorphisms 
\begin{equation*}
    h:\mathrm{Res}_{\bbC/\bbR}\bbG_m\rightarrow G_{\bbR}
\end{equation*}
such that $(G,X)$ is a Shimura datum. Given a compact open subgroup $K\subset G(\bbA_f)$, we can form the associated Shimura variety
\begin{equation*}
    \mathrm{Sh}_K(G,X)(\bbC)=G(\bbQ)\backslash X\times G(\bbA_f)/ K,
\end{equation*}
which is a complex manifold for $K$ small enough, and has a canonical model over $E$, for some number field $E$.

Consider a representation valued in complex vector spaces
\begin{equation*}
    \xi:G_\bbC\rightarrow \GL(V_\xi)
\end{equation*}
such that $\xi(Z(\bbQ)\cap K)=1$, where $Z$ is the center of $G$. This gives rise to a locally constant sheaf of complex vector spaces 
\begin{equation*}
    \calL_\xi=G(\bbQ)\backslash V_\xi\times X\times G(\bbA_f)/ K.
\end{equation*}

Suppose now that $G^\der$ is anisotropic, so that the Shimura variety $\Sh_K(G,X)$ is compact. This assumption is not necessary for the statements of the theorems, as long as we replace the cohomology groups with the intersection cohomology; we assume this here for simplicity. We now consider the complex analytic cohomology of the tower of the Shimura variety, i.e.
\begin{equation*}
    H^i(\Sh(G,X)^{\an},\calL_\xi):=\lim_{\substack{\to \\ K}}H^i(\Sh_K(G,X)^{\an},\calL_\xi).
\end{equation*}
Let $A_G$ be the largest $\mathbb{Q}$-split subtorus of $Z_G$. Choose $h\in X$, and let $K_\infty$ be the stabilizer of $h$ in $G(\bbR)$. By Matsushima's formula for $L^2$-cohomology established by Franke \cite{Franke}, and (the proof of) Zucker's conjecture, we have a decomposition of $H^i(\Sh(G,X)^{\an},\calL_\xi)$ in terms of Lie algebra cohomology 
\begin{equation*}
    H^i(\Sh(G,X)^{an},\calL_\xi)=\bigoplus_{\pi=\pi_\infty\otimes\pi^\infty}m(\pi)H^i(\mathfrak{a}_G\backslash\frakg,K_\infty;\pi_\infty\otimes\xi)\otimes (\pi^\infty),
\end{equation*}
where $\mathfrak{a}_G=\Lie(A_G)$, and $\pi$ runs through unitary automorphic representations of $G(\bbA)$, and $m(\pi)$ is the multiplicity of $\pi$ appearing in $L^2(G(\bbQ)\backslash G(\bbA),\omega)$, the space of measurable functions on $G(\bbQ)\backslash G(\bbA)$ which are square-integrable modulo center and with infinitesimal character $\omega$.

Fix a prime $\ell$, and an isomorphism $\iota:\overline{\bbQ}_\ell\xrightarrow{\sim}\bbC$. The representation $\xi$ gives rise to an $\ell$-adic automorphic sheaf, in the following way. $\xi$ gives rise to a representation valued in $\overline{\bbQ}_\ell$-vector spaces
\begin{equation*}
    \xi_\ell:G_{\overline{\bbQ}_\ell}\rightarrow \GL(V_{\xi,\ell}),
\end{equation*}
and similarly gives rise to a sheaf
\begin{equation*}
    \calL_{\xi,\ell}=G(\bbQ)\backslash V_{\xi,\ell}\times X\times G(\bbA_f)/ K.
\end{equation*}
Since $\Sh_K(G,X)$ is defined over $E$, we can consider the \'{e}tale cohomology of the tower 
\begin{equation*}
    H^i_{\acute{e}t}(\Sh(G,X)_{\bar{E}},\calL_{\xi,\ell}):=\lim_{\substack{\to \\ K}}H^i_{\acute{e}t}(\Sh_K(G,X)_{\bar{E}},\calL_{\xi,\ell}),
\end{equation*}
then we have an isomorphism 
\begin{equation*}
    H^i_{\acute{e}t}(\Sh(G,X)_{\bar{E}},\calL_{\xi,\ell})\simeq H^i(\Sh(G,X)^{\an},\calL_\xi)\otimes_\iota \overline{\bbQ}_\ell,
\end{equation*}
which is $G(\bbA_f)$-equivariant, and thus we have a decomposition
\begin{equation*}
    H^i_{\acute{e}t}(\Sh(G,X)_{\bar{E}},\calL_{\xi,l})=\bigoplus_{\pi^\infty}V^i(\pi^\infty)\otimes (\pi^\infty),
\end{equation*}
where $V^i(\pi^{\infty})=\mathrm{Hom}_{G(\bbA_f)}(\pi^\infty,H^i_{\acute{e}t}(\Sh(G,X)_{\bar{E}},\calL_{\xi,\ell}))$. If $\Sh_K(G,X)$ is not compact, we consider instead the intersection cohomology of the Bailey-Borel compactification $\Sh(G,X)^{\min}$,
\[V^i(\pi^{\infty})=\mathrm{Hom}_{G(\bbA_f)}(\pi^\infty,IH^i_{\acute{e}t}(\Sh(G,X)^{\min}_{\bar{E}},\IC(\calL_{\xi,\ell}))).\]
We now let $\sigma_{\pi^\infty}$ be the associated Galois representation
\begin{equation*}
    \sigma_{\pi^\infty}:\mathrm{Gal}(\bar{E}/E)\rightarrow V^i(\pi^\infty).
\end{equation*}
On the other hand, the global Langlands correspondence conjectures that to the automorphic representation $\pi$ we have an associated semisimple Galois representation $\phi_\pi:\mathrm{Gal}(\bar{E}/E)\rightarrow {^L}G$, and 
we denote the composition
\begin{equation*}
    \rho_{\pi,\mu}:\mathrm{Gal}(\bar{E}/E)\xrightarrow{\phi_\pi} {^L}G\xrightarrow{r_{-\mu}} \GL(V_{-\mu}),
\end{equation*}
where $\mu$ is the minuscule Hodge cocharacter associated to the Shimura datum, and $V_{-\mu}$ is the vector space of the associated highest weight representation $r_{-\mu}$. Here, $r_{-\mu}$ is the highest weight representation associated with the dominant Weyl conjugate of $\mu$. We remark here that this $-\mu$ is such that the action of the geometric Frobenius at the unramified places corresponds to the \emph{left} Hecke action of $\mu^{-1}(p)$ under Artin reciprocity.

When the group $G$ is of the form $\mathrm{Res}_{F/\bbQ}H$ for some connected reductive group $H$, and $F$ is a totally real field of degree $d$, the global Langlands correspondence then further conjectures that we have a decomposition (perhaps after passing to a finite extension $E'$ of $E$) of $\tilde{\rho}$ as
\begin{equation*}
    \rho_{\pi,\mu}=\bigotimes_{v|\infty}\rho_{\pi,\mu,v},
\end{equation*}
where the product runs over infinite places of $F$. Moreover, $\rho_{\pi,\mu,v}$ should have the following form. Observe that over $\bbC$, we have a decomposition
\begin{equation*}
    \mathrm{Res}_{F/\bbQ}H\simeq \prod_{v|\infty} H,
\end{equation*}
and the cocharacter $\mu$ also decomposes as $\prod_v \mu_v$. Then we should obtain $\rho_{\pi,\mu,v}$ as the composition 
\begin{equation*}
    \rho_{\pi,\mu,v}=\rho_{\pi,\mu_v}: \mathrm{Gal}(\bar{E}'/E')\xrightarrow{\phi_\pi} {^L}H\xrightarrow{r_{-\mu_v}} \GL(V_{-\mu_v}),
\end{equation*}
where here we view $\pi$ as an automorphic representation of $H$ over $F$.

We expect a close relationship between $\sigma_{\pi^\infty}$ and $\rho_{\pi,\mu}$. In the case where $\pi$ does not have endoscopy (for instance, if it is a twist of Steinberg at some finite place), then we expect to have
\[\sigma_{\pi^\infty}\simeq \rho_{\pi,\mu}^{\oplus m}\]
for some integer $m$, and thus if $\rho_{\pi,\mu}$ is irreducible, $\sigma_{\pi^\infty}$ semisimple. Such a result may be shown via the Langlands-Kottwitz method, in combination with the results of this paper. 

The main theorem of this paper is the following:
\begin{theorem}
\label{thm:main}
Let $(G,X)$ be a Shimura datum of abelian type such that $G=\mathrm{Res}_{F/\bbQ}H$ for some connected reductive group $H$, and totally real number field $F$. Let $\pi$ be an automorphic representation of $G(\bbA_{\bbQ,f})=H(\bbA_{F,f})$. For all $v$, suppose that the $^{L}H$-valued Galois representation associated to $\pi$ exists, and we consider for all $v$ the composition with the highest weight representation $\rho_{\pi,\mu_v}:\mathrm{Gal}(\bar{F}/F)\rightarrow {^L}H\rightarrow \GL(V_{-\mu_v})$. Suppose that moreover we also know that 
\begin{enumerate}
    \item $\rho_{\pi,\mu_v}$ is strongly irreducible
    \item For all primes $\frakp$ of $E$ such that $\frakp|l$, the Hodge-Tate weights of $\rho_{\pi,\mu_v}|_{D_\frakp}$ are distinct.
\end{enumerate}
Then $\sigma_{\pi^\infty}$ is a semisimple representation.
\end{theorem}

This result was previously known in the cases where the abelian type Shimura variety had a cover by an associated PEL type A Shimura variety, as shown in \cite{FN19}. Note that if $G$ is adjoint, then the group $G$ is always a product of groups of the form $\mathrm{Res}_{F/\mathbb{Q}}H$ for some absolutely simple adjoint group $H$.

\begin{remark}
    The statement of Theorem \ref{thm:main} is most interesting when $F\neq\bbQ$, and $H$ is absolutely simple. Indeed, while the statement is still valid even when $F=\bbQ$, in this situation it is often the case that the Galois representations $\phi_\pi$ is constructed in the cohomology of $\Sh_K(G,X)$, and thus the statement is tautological. Moreover, the condition on $\rho_{\pi,\mu_v}$ to be strongly irreducible is rarely satisfied if $H$ is not absolutely simple.
\end{remark}
We can apply the above theorem to situations beyond the cases of unitary groups considered in \cite{FN19}. In particular, we can consider now the cases where $H$ is an inner form of the groups $\GSp_{2g}$ or $\mathrm{GSO}_{2n}$, where the automorphic Galois representations have been constructed by  Kret and Shin in \cite{KS2020galois} and \cite{KS2020}. In the following theorem, we will let $G^*$ be one of the following groups:
\begin{description}
\item[symplectic] $G^*=\GSp_{2g}$
\item[orthogonal, $n$ even]  $G^*=\mathrm{GSO}_{2n}$
\item[orthogonal, $n$ odd] $G^*$ is a non-split quasi-split form of $\mathrm{GSO}_{2n}$ relative to $E/F$, a CM extension
\end{description}
We now assume that $H$ is an inner form of $G^*$. Let $d$ be the dimension of $\Sh(G,X)$. 
\begin{theorem}
\label{thm:GSpapplication}
Let $\pi$ be a cuspidal $L$-algebraic automorphic representation of $H(\bbA_F)$, satisfying
\begin{enumerate}
    \item There is a finite $F$-place $v_{St}$ such that $H_{F_{v_{St}}}$ and $G^*_{F_{v_{St}}}$ are isomorphic, and under this isomorphism $\pi_{v_{St}}$ is the Steinberg representation of $G^*(F_{v_{St}})$ twisted by a character.
    \item $\pi_\infty|\mathrm{sim}|^{-n(n-1)/4}$ is $\xi$-cohomological for an irreducible algebraic representation $\xi = \otimes_{y:F\rightarrow \bbC} \xi_y$ of the group $(\mathrm{Res}_{F/\bbQ}G^*)_\bbC$, where $\mathrm{sim}$ is the similitude factor map $\mathrm{sim}: G^* \rightarrow \bbG_m$.
    \item The representation $\pi_v$ is regular after composing with the representation $\mathrm{GSpin}_{2g+1}\xrightarrow{\mathrm{spin}}\GL_{2^g}$ if symplectic (resp. $\mathrm{GSpin}_{2n} \xrightarrow{\mathrm{std}} \GL_{2n}$ if orthogonal) at every infinite place $v$ of $F$.
\end{enumerate} 
If moreover the $\ell$-adic Galois representation $\phi_\pi:\mathrm{Gal}(\bar{F}/F)\rightarrow \widehat{H}$ satisfies
\begin{enumerate}
\setcounter{enumi}{3}
    \item The image of $\phi_\pi$ is Zariski dense in $\widehat{H}$,
\end{enumerate}
Then the Galois module 
\begin{equation*}
    \mathrm{Hom}_{G(\bbA_f)}(\pi^\infty,IH_{\acute{e}t}^d(\Sh(G,X)^{\min}_{\bar{E}},\IC(\calL_{\xi,\ell})))
\end{equation*}
is semisimple.
\end{theorem}

To show this result, we first define a `partial Frobenius isogeny' at primes $p$ which are split in $F$, and then show the Eichler-Shimura congruence relations for these partial Frobenius for split groups, using results from \cite{Lee20}. More precisely, we show the following:
\begin{theorem}\label{thm:ESpartialfrob}
    Let $(G,X)$ be a Shimura datum of abelian type, such that $G=\mathrm{Res}_{F/\bbQ}H$ for some connected reductive group $G$, and totally real number field $F$ of degree $d$. Let $p$ be a prime which is split in $F$ and for which the group $G_{\bbQ_p}$ is split. Then for all $i=1,\dots, d$ we have a partial Frobenius correspondence $\Frob_{\frakp_i}$ such that
    \begin{equation*}
        \Frob=\prod_i \Frob_{\frakp_i}
    \end{equation*}
    where $\frakp_i$ is a prime of $F$ dividing $p$, and in the ring of algebraic correspondences of the mod $p$ reduction of $\Sh_K(G,X)$,
    \begin{equation*}
        H_{i}(\Frob_{\frakp_i})=0,
    \end{equation*}
    where $H_i$ is the Hecke polynomial at $\frakp_i$, which is a renormalized characteristic polynomial of the irreducible representation of $\widehat{H}$ with highest weight $\hat{\mu}_i$.
\end{theorem}
This is a refinement of the Eichler-Shimura relation considered for Hodge type Shimura varieties in \cite{Lee20}. The Hecke polynomial $H_{G,X}$ defined there (which is valid for all Shimura varieties) will be a tensor product of the polynomials $H_i$, and in particular, under the same assumptions of splitting of the group $G$ at the prime $p$, we get the Eichler-Shimura relations for the abelian type Shimura varieties considered in the theorem.

We also remark that the construction of this partial Frobenius isogeny should be viewed as a shadow of the plectic conjecture of Nekov\'{a}\v{r}-Scholl, and while in this paper we construct it only with the additional assumption that $G_{\bbQ_p}$ is split, in fact the construction does not need this assumption.

Once we have the generalized Eichler-Shimura relations, we can combine the above theorem with a semisimplicity criterion for Lie algebras shown in \cite{FN19}, which allows us to deduce the main result. Roughly speaking, the generalized Eichler-Shimura relations, along with the condition of distinct Hodge-Tate weights, shows that we have many Frobenius elements, in particular a positive density of them, which are semisimple. This, together with a strong condition on irreducibility of the Galois representation coming from the Langlands correspondence, allows us to show that the image of Galois is a reductive group, and hence the representation was semisimple.

We emphasize that the main theorem here \emph{cannot} be used to prove semisimplicity of the Galois representations constructed in, for instance, \cite{BLGGT}, \cite{KS2020} or \cite{KS2020galois}. Instead, the main novelty of this result is when one knows, for instance from \cite{KSZ21}, the expected shape of the semisimplification of the Galois representation appearing in the cohomology of the Shimura variety, via the study of zeta functions of Shimura varieties and the Langlands-Kottwitz method. In this case, we can conclude that the equality holds even without taking semisimplification.  
\subsection*{Acknowledgements}
Many thanks to my advisor, Mark Kisin, for introducing this problem to me, and for many helpful conversations. Thanks also to Sug Woo Shin, for suggestions about applications in Section 4. I would also like to thank the referee for detailed feedback, corrections and suggestions.

\section{Eichler-Shimura Relations}
In this section, we will prove the key technical result Theorem \ref{thm:ESpartialfrob}, and the corresponding result on (intersection) cohomology. 

For this entire section, we will fix $(G,X)$ a Shimura datum of abelian type such that $G =\mathrm{Res}_{F/\bbQ} H$ with $F$ totally real of degree $d$. Choose a prime $p>2$ that is split in $F$ such that $G_{\bbQ_p}$ is split. Let $p=\frakp_1\dots\frakp_d$, for $\frakp_i$ primes of $F$ above $p$. In this case, we have
\[G_{\bbQ_p} = \prod_{\frakp_i|p} H_i.\]
Fix a choice of Borel $B$ and torus $T$ of $G_{\bbQ_p}$. 

For any $h\in X$, we can define
\[\mu_h:\bbC^\times\xrightarrow{x\mapsto (x,\bar{x})} \mathbb{C}^\times \times \mathbb{C}^\times=\mathbb{S}\rightarrow G(\mathbb{C}).\]
The conjugacy class $\{\mu_h\}$ is defined over the reflex field $E$. Let $v|p$ be a prime of $E$. In the setup above where $G_{\bbQ_p}$ is split, we thus get that the conjugacy class of cocharacters $\{\mu\}$ obtained from $X$ is defined over $\bbQ_p$, and thus in this case we have $E_v=\mathbb{Q}_p$. Let $\mu\in X_*(T)$ be a dominant representative of this conjugacy class, and the decomposition of $G_{\bbQ_p}$ means we can write $\mu=\prod_i\mu_i$.

Fix a prime $\ell\neq p$. Let $K\subset \bbA_f$ be an open compact subgroup which we fix for this entire paper, and for which $K_p$ is hyperspecial. Let $\mathbb{T}^p_f$ denote the prime-to-$p$-adelic Hecke algebra $\calH(G(\bbA^p_f)//K^p,\bbQ)$ of $K^p$-bi-invariant compactly supported functions on $G(\bbA^p_f)$. 

Let $\xi:G_{\Qell}\rightarrow \GL(V_\xi)$ be a representation such that $\xi(Z(\bbQ)\cap K)=1$. As explained in the introduction, this gives rise to an $\ell$-adic local system $\calL_{\xi,\ell}$ on $\Sh_K(G,X)$.

For a choice of hyperspecial $K_p$ which decomposes as $K_p=\prod_i K_{p,i}$ we get for each $i$ Hecke polynomials $H_{\mu_i} \in \mathcal{H}_i:=\mathcal{H}(H_i(\bbQ_p)//K_{i,p},\bbQ)$ the Hecke algebra of $K_{ip}$ bi-invariant functions on $H_i(\bbQ_p)$. We define $H_{\mu_i}$ in \ref{subsect:Heckealg}. We also define the Hecke algebra at $p$, $\calH_p:=\mathcal{H}(G(\bbQ_p)//K_{p},\bbQ)$.

The main result of this section is the following:
\begin{theorem}\label{ref:mainthms2}
There are operators $\mathrm{Frob}_{\frakp_i}$ for $i=1,\dots,d$ acting on $IH^j(\Sh_K(G, X)^{\min}_{\bar{E}},\IC(\calL_{\xi,\ell}))$ such that
\begin{enumerate}
    \item The action of any pair $\Frob_{\frakp_i}$, $\Frob_{\frakp_j}$ commute, as does the action of any $\Frob_{\frakp_i}$ with the $\mathbb{T}^p_f$-action and the $\mathcal{H}_p$-action;
    \item $\Frob_{\frakp_1}\circ\dots\circ \Frob_{\frakp_d}=\Frob_v$, where $\Frob_v$ is the action of the geometric Frobenius at $v$;
    \item For each $i$, $H_{\mu_i}(\Frob_{\frakp_i})$ acts as 0 on $IH^j(\Sh_K(G, X)^{\min}_{\bar{E}},\IC(\calL_{\xi,\ell}))$.
\end{enumerate}
\end{theorem}
In the theorem, $IH^j(\Sh_K(G, X)^{\min}_{\bar{E}},\IC(\calL_{\xi,\ell}))$ denotes the (\'{e}tale) intersection cohomology of the Bailey-Borel compactification $\Sh_K(G, X)^{\min}$ of $\Sh_K(G, X)$, with coefficients in the perverse intermediate extension of $\calL_{\xi,\ell}$ to $\Sh_K(G, X)^{\min}$. More precisely, if we let $j:\Sh_K(G,X)\hookrightarrow \Sh_K(G,X)^{\min}$ be the open embedding, then intermediate extension gives us a perverse $\ell$-adic sheaf $\IC(\calL_{\xi,\ell}):=j_{!\ast}\calL_{\xi,\ell}$, and $IH^j(\Sh_K(G, X)^{\min}_{\bar{E}},\IC(\calL_{\xi,\ell}))$ is the $j$-th cohomology group of $\Sh_K(G,X)^{\min}$ with coefficients in $\IC(\calL_{\xi,\ell})$.
\subsection{Integral model}
We first recall the construction in \cite[\S3]{K2010}, of the integral model $\sS_K(G,X)$ of $\Sh_K(G,X)$ over $O_{E_v}$. Let $(G_1,X_1)$ denote a Hodge type Shimura datum which covers $(G,X)$, i.e. such that there exists a central isogeny
\begin{equation*}
    f:G_1^{\der}\rightarrow G^{\der}
\end{equation*}
which induces an isomorphism on adjoint Shimura data $(G_1^{\ad},X_1^{\ad})\simeq (G^{\ad},X^{\ad})$. In fact, we may choose $(G_1,X_1)$ as in the following lemma:
\begin{lemma}
There exists a covering Hodge type datum $(G_1,X_1)$ such that the local reflex field $E_{1,v'}$ of $(G_1,X_1)$ for some prime $v'|p$ is $\bbQ_p$. 
\end{lemma}
\begin{proof}
    Since $G_{\bbQ_p}$ is split this implies that $G^\ad$ is also split over $\bbQ_p$. We can now apply \cite[Lemma 4.6.22]{KP2018}. More precisely, in the proof of loc. cit. it is shown that one may choose a covering Shimura datumn $(G_1,X_1)$ of Hodge type such that any prime $v|p$ of the reflex field of $(G^\ad,X^\ad)$, denoted $E(G^\ad,X^\ad)$, splits in $E(G_1,X_1)$. In particular, since $E(G^\ad,X^\ad)_v=\bbQ_p$, we may in fact choose any prime $v'$ over $v$ and we will have $E(G_1,X_1)_{v'}=\bbQ_p$.  
\end{proof}   

Fix a connected component $X^+\subset X$. Let $\Sh_K(G,X)^+$ denote the geometric connected component of $\Sh_K(G,X)$ containing $\{1\}\times X^+$. Consider now the geometric connected component
\[\Sh_{K_{1p}}(G_1, X_1)^+ = \lim_{\substack{\leftarrow\\K_1^p}}\Sh_{K_{1p}K^p_1} (G_1, X_1)^+\]
of $\Sh_{K_{1p}} (G_1, X_1) = \lim_{\leftarrow {K_1^p}}\Sh_{K_{1p}K^p_1} (G_1, X_1)$, where $K_{1p} = G_1(\bbZ_p)$. Let $\sS_{K_{1p}} (G_1, X_1)$ be the integral model of the Hodge type Shimura variety constructed in \cite[\S2]{K2010}, and let $\sS_{K_{1p}} (G_1, X_1)^+$ be the Zariski closure of $\Sh_{K_{1p}} (G_1, X_1)^+$ in $\sS_{K_{1p}} (G_1, X_1)$. Write $Z = Z_G$, and define
\[\sS_{K_p}(G,X) = [\mathscr{A}(G_{\bbZ_{(p)}})\times \sS_{K_{1p}} (G_1, X_1)^+]/\mathscr{A}(G_{1\bbZ_{(p)}})^\circ,\]
where 
\[\mathscr{A}(G_{\bbZ_{(p)}}) = G(\bbA^p_
f)/Z(\bbZ_{(p)})^- \ast_{G(\bbZ_{(p)})^+/Z(\bbZ_{(p)})} G^{\ad}(\bbZ_{(p)})^+,\]
and
\[\mathscr{A}(G_{\bbZ_{(p)}})^\circ = G(\bbZ_{(p)})^-_
+/Z(\bbZ_{(p)})^- \ast_{G(\bbZ_{(p)})^+/Z(\bbZ_{(p)})}G^{\ad}(\bbZ_{(p)})^+.\]
It is shown in \cite[Theorem 3.4.10]{K2010} that $\sS_{K_p}(G,X)$ as defined above is then the canonical integral model of $\Sh_{K_p}(G,X)$.

For any abelian type Shimura datum unramified at $p$ with level $K$ which is hyperspecial at $p$, we denote by
\[\barSs_{K}(G, X):=\sS_{K}(G, X)_{k_v},\]
where $k_v$ is the residue field of $O_{E_v}$.

We also recall the result of Lan-Stroh \cite[Corollary 4.10]{LanStroh2}, which will allow us to deduce results about cohomology from results on the mod $p$ fiber:
\begin{theorem}
\label{thm:Lanstroh}
There is an isomorphism
\[IH^i(\Sh_K(G, X)^{\min}_{\bar{E}},\IC(\calL_{\xi,\ell}))\simeq IH^i(\barSs_K(G, X)^{\min}_{\bar{k}_v},\IC(\calL_{\xi,\ell})),\]
where on the right-hand side we abuse notation and let $\calL_{\xi,\ell}$ also denote the mod $p$ fiber of the canonical extension of the $\ell$-adic local system $\calL_{\xi,\ell}$ from the generic fiber. This isomorphism is equivariant for the action of the prime-to-$p$ Hecke algebra $\mathbb{T}^p_f$, and also the action of the Frobenius at $p$.
\end{theorem}
\subsection{Newton stratification}\label{subsec:Newtonstrata}
We now recall the construction of Newton strata for Shimura varieties of abelian type, as constructed in \cite[\S2]{ShenZhang2017}. Let $L=W(\bar{\bbF}_p)[1/p]$, with Frobenius given by $\sigma$, and $\upsilon=\sigma(\mu^{-1})=\mu^{-1}$. Recall that $B(G)$ denotes the set of $\sigma$-conjugacy classes of $G(L)$. The Newton strata of the Shimura variety are indexed by elements in $B(G,{\upsilon})$.

The following result of Kottwitz is a key input in the construction.
\begin{lemma}[{\cite[Section 6.5]{Kott97}}]
\label{lemma:B(G)}
    Let $G$ be a connected reductive group over $\bbQ_p$, and let $\{\upsilon\}$ be a conjugacy class of minuscule cocharacters. Then, projecting to the adjoint group $G^\ad$ induces a map $B(G)\rightarrow B(G^\ad)$, and restricting this map to $B(G,\{\upsilon\})$ induces a bijection between $B(G,\{\upsilon\})$ and $B(G^{\ad},\{\upsilon^{\ad}\})$.
\end{lemma}

In \cite[\S2.3]{ShenZhang2017}, the Newton strata are first constructed for the adjoint Shimura variety. More precisely, let $K_p^\ad$ be the hyperspecial subgroup of $G^\ad(\bbQ_p)$ determined by $K_p$. (Equivalently, from the arguments in \cite[Corollary~3.4.14]{K2010}, this is also the hyperspecial subgroup determined by $K_{1p}$.) In \cite[2.3.3,2.3.4]{ShenZhang2017}, it is shown that if we restrict the Newton stratification to the connected component $\barSs_{K_{1p}}(G_1, X_1)^{+} \hookrightarrow \barSs_{K_{1p}}(G_1, X_1)$, and we denote by $\barSs_{K_{1p}}(G_1, X_1)^{[b_1],+}$ the strata corresponding to $[b_1]\in B(G_1,\{\upsilon_1\})$, then taking
\[\barSs_{K^\ad_p}(G^\ad, X^\ad)^{[b^\ad]} = [\mathscr{A}(G^\ad_{\bbZ_{(p)}})\times \barSs_{K_{1,p}}(G_1,X_1)^{[b_1]+}]/\mathscr{A}(G_{1,\bbZ_{(p)}})^\circ\]
gives a well-defined stratification on $\barSs_{K^\ad_p}(G^\ad, X^\ad)$. Here, we let $[b^\ad]$ be the image of $[b_1]$ under the bijection of Lemma \ref{lemma:B(G)}.

We have a natural map of integral models $\sS_{K_p}(G,X)\rightarrow \sS_{K^\ad_p}(G^\ad,X^\ad)$ induced by the map of Shimura data $(G,X)\rightarrow (G^\ad,X^\ad)$ and the extension property, which may moreover be seen explicitly as follows. 
\[\sS_{K_p}(G,X)\simeq [\mathscr{A}(G_{\bbZ_{(p)}})\times \sS_{K_p}(G,X)^+]/\mathscr{A}(G_{\bbZ_{(p)}})^\circ \rightarrow [\mathscr{A}(G^\ad_{\bbZ_{(p)}})\times \sS_{K_p}(G,X)^+]/\mathscr{A}(G_{\bbZ_{(p)}})^\circ,\]
where the last term is isomorphic to $\sS_{K^\ad_p}(G^\ad,X^\ad)$. In particular, if we let 
\[\Delta = \ker(\mathscr{A}(G_{\bbZ_{(p)}})^\circ \rightarrow \mathscr{A}(G^\ad_{\bbZ_{(p)}}));\]
observe that this is a finite group, and we have that $\sS_{K^\ad_p}(G^\ad,X^\ad)^+\simeq\sS_{K_p}(G,X)^+/\Delta$ is a finite \'{e}tale quotient.

The Newton strata for $\barSs_{K_p}(G, X)$ is then defined in \cite[2.3.5]{ShenZhang2017} to be the pullback of the Newton strata for $\barSs_{K^{\ad}_p}(G^{\ad}, X^{\ad})$ via the natural map
\begin{equation*}
    \barSs_{K_p}(G, X)\rightarrow \barSs_{K^{\ad}_p}(G^{\ad}, X^{\ad}).
\end{equation*}
On geometric connected components, the construction implies that we have the following compatibility of Newton strata. In particular, we see that the Newton strata are exactly pulled back along the maps
\begin{equation*}
    \barSs_{K_{1p}}(G_1, X_1)^{+}\rightarrow \barSs_{K_{p}}(G, X)^{+}\rightarrow \barSs_{K^{\ad}_{p}}(G^{\ad}, X^{\ad})^{+}.
\end{equation*}
Moreover, as stated in \cite[2.3.4]{ShenZhang2017} we can also realize the Newton strata for $(G,X)$ from the one for $(G_1,X_1)$ in terms of the action of $\mathscr{A}(G_{\bbZ_{(p)}})$. In fact we have 
\[\barSs_{K_{p}}(G,X)^{[b]} = [\mathscr{A}(G_{\bbZ_{(p)}}) \times \barSs_{K_{1p}}(G_1,X_1)^{[b],+}]/\mathscr{A}(G_{1\bbZ_{(p)}})^\circ.\]

We can take quotients by prime-to-$p$ level $K^p$ to obtain the Newton strata at finite level $K$.

Finally, observe that by construction we have $\mu_1$ is defined over $\bbQ_p$, and thus the ordinary locus of $\barSs_{K_{1p}}(G_1,X_1)$ is non-empty, and thus we will also refer to the maximal Newton strata for the Shimura varieties $\barSs_{K_{p}}(G,X)$ and $\barSs_{K^\ad_{p}}(G^\ad,X^\ad)$ as the ordinary locus, even though these are not parametrizing abelian varieties.
\subsection{Rapoport-Zink spaces for abelian type}
For this entire section, we let $(G,[b],\{\upsilon\})$ be an unramified local Shimura datum, i.e. $G$ is a reductive group over $\bbQ_p$ which has a reductive model over $\bbZ_p$, which we still denote by $G$, $[b]\in B(G,\{\upsilon\})$, and $\{\upsilon\}$ is a conjugacy class of minuscule cocharacters. If $\mu$ is the dominant Hodge cocharacter of Shimura varieties, then in the local Shimura data we should take $\upsilon=\sigma(\mu^{-1})$, as explained in \cite[1.4.1]{K2017}.

\subsubsection{}We first recall the group-theoretic construction of affine Deligne-Lusztig varieties. Let $LG$ and $L^+G$ be the loop group (resp. positive loop group) defined in \cite[\S3.1]{XZ17}, and we consider the affine Grassmannian $\Gr_G$ defined as the \'{e}tale quotient $LG/L^+G$. Fix a representative $b\in[b]$. The affine Deligne Lusztig variety $X^G_\upsilon(b)$ is the perfect closed subvariety inside $\Gr_G$ whose $\bar{\bbF}_p$-points are given by
\[X^G_\upsilon(b)(\bar{\bbF}_p)=\{g\in G(L)/G(W)\mid g^{-1}b\sigma(g) \in G(W)\upsilon(p)G(W)\}.\]
where $W=W(\bar{\bbF}_p)$.
\subsubsection{} We now want to recall results about Rapoport-Zink spaces, following Shen \cite{shen_2020}.
\begin{theorem}[Theorem 4.6,\cite{shen_2020}]
\label{thm:RZconstruct}
Let $(G, [b], \{\upsilon\})$ be an unramified local Shimura datum of abelian type. Fix a representative $b \in G(L)$ of $[b]\in B(G,\{\upsilon\})$. Then there exists a formal scheme $\RZ(G,b,\{\upsilon\})$, which is formally smooth, formally locally of finite type over $W$, such that moreover we have that the perfection of the reduced special fiber is isomorphic to the affine Deligne-Lusztig variety $X_{\upsilon}^G(b)$.
\end{theorem}
In particular, given an element $g\in X_\upsilon^G(b)$, we will in the following also say that $g\in \RZ(G,b,\{\upsilon\})$.

\subsubsection{} Note that from the construction of $\RZ(G,b,\{\upsilon\})$ that if $M\subset G$ is a Levi subgroup which contains $T$, and we chose a representative $b$ such that $b\in M(L)$, then we have a natural inclusion 
\[\RZ(M,b,\{\upsilon\})\hookrightarrow \RZ(G,b,\{\upsilon\})\]
as a closed formal subscheme.

Moreover, there is also a uniformization map on isogeny classes, from \cite[Theorem~6.7]{shen_2020}.
\begin{theorem}
    \label{thm:RZuniform}
    We have an isomorphism of formal schemes
    \[\Theta:\bigsqcup_{\phi,\phi^\ad=\phi_0}I_\phi(\bbQ)\backslash \RZ(G,b,\{\upsilon\})\times G(\mathbb{A}^p_f)\xrightarrow{\sim} \widehat{S_{K_p}}_{/Z_{\phi_0}}.\]
\end{theorem}
The group $I_\phi$ is defined in loc. cit. Here, $\phi_0$ (resp. $\phi$) is an admissible morphism of Galois gerbs for $G^\ad$ (resp. $G$), and $\widehat{S_{K_p}}_{/Z_{\phi_0}}$ is the completion of $\sS_{K_p}(G,X)$ along $Z_{\phi_0}$ a locally closed subspace of $\barSs_{K_p}(G,X)$ defined in loc. cit. which is the union of mod $p$ isogeny classes corresponding to $\phi_0$. Since the exact definition is not needed for the rest of this paper, we do not recall the definition of an admissible morphism of Galois gerbs here, but it is defined in \cite[3.3.6]{K2017}. In this paper it suffices to understand this as an indexing set for isogeny classes. 

\subsubsection{}\label{subsub:uniformbasepoint} Since we would like to understand how this map works with a fixed base-point, let us briefly recall the construction. For the Hodge type Shimura datum $(G_1,X_1)$ and a point $x\in \barSs_{K_p}(G_1,X_1)^b(\bar{\bbF}_p)$ if we choose the representative $b_1 \in [b_1]$ associated to $x$ as in \cite[1.4.1]{K2017}, we have a uniformization map
\[\Theta_{b_1}: \RZ(G_1,b_1,\{\upsilon_1\})\times G_1(\bbA^p_
f ) \rightarrow \widehat{\sS}_{K_{1p}}(G_1,X_1)^{[b_1]},\]
sending the identity element $1\in \RZ(G_1,b_1,\{\upsilon_1\})$ to $x$. In fact, this induces an injective map
\[\Theta_{x,b_1}:I_{\phi_1}(\bbQ)\backslash\RZ(G_1,b_1,\{\upsilon_1\})\times G_1(\bbA^p_
f ) \rightarrow \widehat{\sS}_{K_{1p}}(G_1,X_1)^{[b_1]},\]
where $\widehat{\sS}_{K_{1p}}(G_1,X_1)^{[b_1]}$ is the completion of $\barSs_{K_{1p}}(G_1,X_1)^{[b_1]}$ along $\sS_{K_{1p}}(G_1,X_1)$. 

Then, as described in \cite[\S6.3]{shen_2020}, there is a relation between the isogeny classes of $G_1$ and $G$, which can be upgraded into a uniformization map for $G$. For $\phi$ an admissible morphism of Galois gerbs for $G$, we define the formal scheme
\[\widehat{S}(G,\phi):=I_{\phi}(\bbQ)\backslash\RZ(G,b,\{\upsilon\})\times G(\bbA^p_
f ),\]
and similarly for $G_1$. Fix $\phi_0$ an admissible morphism of Galois gerbs for $G^\ad$. Let 
\[\widehat{S}(G, \phi_0) = \bigsqcup_{
[\phi],\phi^{\ad}=\phi_0}
\widehat{S}(G,\phi).\]
Denote by $\widehat{S}(G_1,\phi_1)^+$ the fiber of $\Theta_{x,b_1}$ over the connected component $\widehat{\sS}_{K_{1p}}(G_1,X_1)^{[b_1]+}$ (note that this is just the open and closed formal subscheme whose mod $p$ fiber lies in $\barSs_{K_{1p}}(G_1,X_1)^+$.), and similarly for $\widehat{S}(G_1,\phi_0)^+$. From \cite[Proposition 6.6]{shen_2020} we know that
\[\widehat{S}(G,\phi_0)\simeq [\mathscr{A}(G_{\bbZ_{(p)}})\times \widehat{S}(G_1,\phi_0)^+]/\mathscr{A} (G_{1\bbZ_{(p)}})^\circ.\]
Now, given a point $x\in\barSs_{K_p}(G,X)^{[b]}(\bar{\bbF}_p)$, from \cite[2.2.5]{K2010} we may act on it by some element $g^p\in G(\bbA^p_f)$ such that $g^px$ lies in $\barSs_{K_p}(G,X)^{[b]+}(\bar{\bbF}_p)$. Thus, there exists some $x'\in \bar{\calS}_{K_1}(G_1,X_1)^{[b_1]+}(\bar{\bbF}_p)$ lifting $g^px$. We thus have a map with base point $x'$
\[\Theta^+_{x',b_1}:\widehat{S}(G_1,\phi_1)^+\hookrightarrow \widehat{\sS}_{K_{1p}}(G_1,X_1)^{[b_1]+},\]
which on taking product with $\mathscr{A}(G_{\bbZ_{(p)}})$ and quotient by $\mathscr{A}(G^\ad_{\bbZ_{(p)}})^\circ$, thus gives us a map 
\[\Theta_{g^px,b}:\widehat{S}(G,\phi)\hookrightarrow \widehat{S}(G,\phi_0)\hookrightarrow\widehat{\sS}_{K_{p}}(G,X)^{[b]},\]
for some $\phi$ such that $g^px$ lies in $Z_{\phi}$. Now, we can translate by $(g^{p})^{-1}$ and quotient by $K^p$ to get an injective map
\[\Theta_{x,b}:\widehat{S}(G,\phi)\rightarrow \widehat{\sS}_{K}(G,X)^{[b]}.\]
\subsubsection{}\label{subsub:extradependencies} Note here that this map depends on some auxillary choices that we fix once and for all when we are defining the Rapoport-Zink space for $(G,b,\{\upsilon\})$ from $(G_1,b_1,\{\upsilon_1\})$, where $b_1$ is as defined in \ref{subsub:uniformbasepoint}.  More precisely, what we see is that from the proof of \cite[Theorem 4.6]{shen_2020} if we let $\RZ(G_1^\ad,b_1^\ad,\{\upsilon_1^\ad\})^+$ denote the connected component of $\RZ(G_1^\ad,b_1^\ad,\{\upsilon_1^\ad\})$ containing $1$, we have an isomorphism
\[\RZ(G_1^\ad,b_1^\ad,\{\upsilon_1^\ad\})^+\simeq \RZ(G,b,\{\upsilon\})^+,\]
of connected components containing $1$, and also an isomorphism
\[\RZ(G,b,\{\upsilon\})\simeq (J_{b}(\bbQ_p)\times \RZ(G_1^\ad,b_1^\ad,\{\upsilon_1^\ad\})^+)/J_b(\bbQ_p)^+\] 
where $J_b(\bbQ_p)^+\subset J_b(\bbQ_p)$ is the stabilizer of the connected component containing 1 in $X^G_{\upsilon}(b)$.

Thus, if we let $g$ be the image of $1\in \RZ(G_1^\ad,b_1^\ad,\{\upsilon_1^\ad\})^+$ under these isomorphism, observe that the map $\Theta_{x,b}$ sends $(g,1)\in \RZ(G,b,\{\upsilon\})\times G(\bbA)_f^p$ to $x$. Moreover, if we replace $b$ by $b'=g'^{-1}b\sigma(g')$, then the map $\Theta_{x,b'}$ instead sends $(g(g')^{-1},1)\in \RZ(G,b',\{\upsilon\})\times G(\bbA_f^p)$ to $x$. 

Observe now that we can compare the $\Theta_{x,b}$ maps for different base-points. For the next proposition, and in the rest of the paper, we will use $\Theta_{x,b}(g)$ to mean $\Theta_{x,b}(g,1)$, that is, we implicitly set the second element in $G(\bbA_f^p)$ to be $1$.
\begin{lemma}
    \label{lemma:Theta}
    Let $g'\in X_\upsilon^G(b)$, and denote by $y=\Theta_x(g')$. Then if we let $b'=(g')^{-1}b\sigma(g')$, we have for all $h\in X_\upsilon^G(b')$, we have
    \[\Theta_{y,b'}(h)=\Theta_{x,b}(hg').\] 
\end{lemma}

\subsubsection{} \label{subsub:RZadic}Consider now the adic generic fiber $\mathcal{M}$ of $\RZ(G,b,\{\upsilon\})$. Observe that the uniformization map in Theorem \ref{thm:RZuniform} implies that we have an isomorphism of adic spaces
\begin{equation}
\label{eqn:uniformization}
\bigsqcup_{\phi,\phi^\ad=\phi_0}I_\phi(\bbQ)\backslash \mathcal{M}\times G(\mathbb{A}^p_f)\xrightarrow{\sim} \calS_{\phi_0}.
\end{equation}
where $\calS_{\phi_0}$ denotes the adic generic fiber of $\widehat{\sS_{K_p}}_{/Z_{\phi_0}}$. As constructed in \cite[Theorem 5.20]{shen_2020}, $\mathcal{M}=\mathcal{M}_{K_p}$ fits into a tower $(\mathcal{M}_{K'_p})$ for $K'_p\subset G(\bbQ_p)$ compact open, and there is an action of $G(\bbQ_p)$ on this tower.

Moreover, from \cite[Corollary 6.10]{shen_2020} for any open compact subgroup $K'_p\subset G(\bbQ_p)$ we also have an isomorphism of adic spaces
\[\bigsqcup_{\phi,\phi^\ad=\phi_0}I_\phi(\bbQ)\backslash \mathcal{M}_{K'_p}\times G(\mathbb{A}^p_f)\xrightarrow{\sim} \calS_{\phi_0,K'_p},\]
where $\calS_{\phi_0,K'_p} \rightarrow (\Sh_{K'_p}(G,X))^\ad$ is the pullback of $\calS_{\phi_0} \rightarrow (\Sh_{K_p}(G,X))^\ad$ under the projection $(\Sh_{K'_p}(G,X))^\ad\rightarrow (\Sh_{K_p}(G,X))^\ad$. In particular, this implies that on the generic fiber the map $\Theta_x$ is equivariant for the Hecke action of $G(\bbQ_p)$ on $\mathcal{M}$ and the usual Hecke action on $(\Sh_{K_p}(G,X))^\ad$.

\subsubsection{} \label{subsub:actionZG}Now, observe that $\RZ(G,b,\{\upsilon\})$ has by constrution an action of $J_b(\bbQ_p)$, which on the reduced special fiber agrees with the action of $J_b(\bbQ_p)$ on $X_\upsilon^G(b)$. In particular, since $Z_G(\bbQ_p)\subset J_b(\bbQ_p)$, we have an action of central elements. Moreover, we can see that action agrees with the Hecke action of $Z_G(\bbQ_p)$ on $\mathcal{M}$, since from \cite[Theorem 5.20]{shen_2020} we can check this after passing to the associated diamond $\mathcal{M}^\diamond$, which is isomorphic to the moduli space of local shtukas $\mathrm{Sht}_{K_p}(G,b,\{\upsilon\})$.

Moreover, we see that $\sS_K(G,X)$ also has an action by $Z_G(\bbQ_p)$, acting by Hecke correspondences, which also induces an action on $\widehat{\sS}_{K}(G,X)$. We have the following:
\begin{proposition}
    \label{prop:ZGaction}
    The map $\Theta_{x,b}$ is equivariant for the action of $Z_G(\bbQ_p)$.
\end{proposition}
\begin{proof}
    Since on both sides $z\in Z_G(\bbQ_p)$ acts as isomorphisms, it suffices to check this equivariance on the generic fibers, where both are constructed via the Hecke action. 
\end{proof}
\subsection{Quasi-isogenies for abelian type}
We now review the meaning of a quasi-isogeny at $p$ for points lying on a Shimura variety of abelian type. We will only define it on characteristic zero or $\bar{\bbF}_p$ valued points, using the Rapoport-Zink spaces defined above and we avoid in this paper the delicate question of what an isogeny means outside this setting.

In characteristic zero, for any $g\in G(\bbQ_p)$, we may consider the covering map from
\[\pi_{K_p}:S_{K^p}(G,X)\rightarrow S_{K}(G,X),\]
given by taking limit over all $K_p$ (trivializing the structures at $p$), thus if $y$ is the preimage of any $x\in S_{K}(G,X)$, we may look at $x':=\pi_{K_p}(gy)$. In such a case we say $x$, $x'$ are isogenous with quasi-isogeny given by $g$. Note that there are only finitely many such $x'$ isogenous to $x$ via $g$.

In characteristic $p$, we use the construction of the uniformization map above. More precisely, for any $x\in \barSs_K(G,X)^{[b]}(\bar{\bbF}_p)$ we consider the uniformization map $\Theta_{x,b}$ considered above to construct the isogeny class. Suppose that $\Theta_{x,b}(g)=x$. Then, for any $g'\in X_\upsilon^G(b)(\bar{\bbF}_p)$, we look at the image $x':=\Theta_{x,b}(g')$, and we say $x,x'$ are isogenous via $(g')^{-1}g$.

In general, we cannot say much about the relation between the quasi-isogeny in characteristic zero and the quasi-isogeny in characteristic $p$, but if $z\in Z_G(\bbQ_p)$ we can do so, from Proposition \ref{prop:ZGaction}, which gives us the following Corollary:
\begin{corollary}
    Let $x\in \barSs_{K_p}(G,X)^{[b]}(\bbF_p)$, and $x'$ a lift of $x$. Suppose that we have $\Theta_{x,b}(g,1)=x$ for some $g\in X_\upsilon(b)$. Then, for any $z\in Z_G(\bbQ_p)$ the mod $p$ reduction of $zx'$ is $\Theta_{x,b}(zg,1)$.
\end{corollary}

\subsection{Models for Hecke correspondences}
For Shimura varieties of Hodge type, the main integral model we consider for Hecke correspondences is $p-\Isog$, a moduli space of $p$-power (quasi)-isogenies preserving extra structures, which was defined and studied in \cite[\S6]{Lee20}. However, for general abelian type Shimura varieties, we do not have a moduli interpretation in terms of abelian varieties, and thus it is unclear how to define $p-\Isog$ in this setting. (This, however, will be settled in forthcoming work with Keerthi Madapusi.) 

Thus, in this paper we work entirely with middle dimensional algebraic cycles in $\sS_{K_p}(G,X)\times \sS_{K_p}(G,X)$ or $\Sh_{K_p}(G,X)\times \Sh_{K_p}(G,X)$ and view Hecke correspondences as elements in this middle dimensional Chow group. We denote the space of such correspondences on the generic fiber, the integral model and the special fiber as $\mathrm{Cor}(\Sh_{K_p}(G,X),\Sh_{K_p}(G,X))$, $\mathrm{Cor}(\sS_{K_p}(G,X),\sS_{K_p}(G,X))$ and $\mathrm{Cor}(\barSs_{K_p}(G,X),\barSs_{K_p}(G,X))$ respectively. Note that all these groups have a multiplicative structure induced by intersection product, as detailed in \cite[Appendix A]{B2002}.
\begin{remark}
    We use the convention here that composition is from right to left, for coherence with our convention that the Hecke action is on the left for the Shimura variety.
\end{remark}

By construction, for any such algebraic correspondence $C$, we have proper maps $p_1,p_2:C\rightarrow \Sh_{K_p}(G,X)$ given by projection onto the first (resp. second) factor, and similarly for the integral model/special fiber. 

Let $1_{K_pgK_p}$ be a double coset in $\calH_p$, and we let $C_g\subset \mathrm{Sh}_{K_p}(G,X)\times \mathrm{Sh}_{K_p}(G,X)$ be the associated Hecke correspondence, which is given as the image of the projection map 
\[\Sh(G,X)\xrightarrow{(\pi\circ g,\pi)}\Sh_{K_p}(G,X)\times \Sh_{K_p}(G,X).\] By construction, this consists of all pairs of $(x',x)$ with a quasi-isogeny from $x$ to $x'$ given by $g$. Note that this cycle depends only on $K_pgK_p$ and not on the choice of $g$, but we use $C_g$ for notational simplicity. Note again that we take $(\pi\circ g,\pi)$ rather than $(\pi,\pi\circ g)$ because all our actions are left actions.

The same argument as in \cite[\S4.3]{B2002} show that this association defines an algebra homomorphism
\begin{equation*}
h:\calH(G(\bbQ_p)//G(\bbZ_p),\bbQ)\rightarrow \mathrm{Cor}(\Sh_{K_p}(G,X),\Sh_{K_p}(G,X)),
\end{equation*}
and by taking Zariski closures and restricting to the special fiber, we get an induced homomorphism, which we also denote by $h$,
\begin{equation}
\label{eqn:defh}
h:\calH(G(\bbQ_p)//G(\bbZ_p),\bbQ)\rightarrow \mathrm{Cor}(\barSs_{K_p}(G,X),\barSs_{K_p}(G,X)).
\end{equation}
Denote by $\mathscr{C}_g$ the Zariski closure of $C_g$ in $\sS_K(G,X)\times \sS_{K_p}(G,X)$. For any correspondence $\mathscr{C}_g$ over $\sS_{K_p}(G,X)$, we will let $C_{g,0}$ denote the special fiber, which is an algebraic correspondence over $\barSs_{K_p}(G,X)$. We have a similar construction for Hecke correspondences for the groups $G_1,G^{\ad}$. Let $g^{\ad}$ denote the image of $g\in G(\bbQ_p)$ in $G^{\ad}$. Let $C_{g^\ad}^{\ad}\subset \mathrm{Sh}_{K_p^{\ad}}(G^{\ad},X^{\ad})\times \mathrm{Sh}_{K_p^{\ad}}(G^{\ad},X^{\ad})$ the associated Hecke correspondence given by $1_{K^{\ad}_pg^{\ad}K^{\ad}_p}$ on the generic fiber. We further denote by $\mathscr{C_{g^\ad}^{\ad}}$ be the closure and let $C^{\ad}_{g^\ad,0}$ denote the special fiber.

\subsubsection{}\label{subsub:comparepIsog}In the Hodge type case for $(G_1,X_1)$ we have an alternative definition of Hecke correspondences, introduced in \cite[\S6]{Lee20} in terms of a moduli space of $p$-power quasi-isogenies, denoted by $p-\Isog$, or $p-\Isog_{G_1}$ when we want to emphasize the group. This is not a scheme, but rather an ind-scheme. Moreover, we also take limits over the prime-to-$p$-level. However, for the purposes of this paper we may take a very large power of $p^n$, and look at the space $p-\Isog_n$ of quasi-isogenies $\rho$ such that $p^n\rho$ is an honest isogeny, and this will be a scheme, though not of finite type. This suffices for us to make sense of irreducible components. As explained in 6.1.6 of loc. cit, for every $1_{K_{1p}g_1K_{1p}}$ we have a union $h(1_{K_{1p}g_1K_{1p}})$ of top-dimensional irreducible components of $p-\Isog$. Moreover, we have proper maps $s,t:p-\Isog\rightarrow \sS_{K_{1p}}(G_1,X_1)$, and thus we may also consider the image $\mathscr{C}'_{g_1}$ of $h(1_{K_{1p}g_1K_{1p}})$ under the map 
\[(s,t):p-\Isog\rightarrow \sS_{K_{1p}}(G_1,X_1)\times \sS_{K_{1p}}(G_1,X_1).\] 
Observe that these two possible Hecke correspondences $\mathscr{C}'_{g_1}=\mathscr{C}_{g_1}$ are the same, since they are both flat and agree on the generic fiber.

\subsubsection{}Observe that the Hecke operators for $G$, $G^{\ad}$ are related as follows. For any $g\in G(\bbQ_p)$, we have that $\mathscr{C}_{g^\ad}^{\ad}$ is the image of $\mathscr{C}_g$ under the (finite \'{e}tale) projection map
\begin{equation*}
    \sS_{K_p}(G,X)\times\sS_{K_p}(G,X) \rightarrow \sS_{K_p^{\ad}}(G^{\ad},X^{\ad})\times \sS_{K_p^{\ad}}(G^{\ad},X^{\ad}).
\end{equation*}
Thus, we see that the map on special fibers $C_{g,0}\rightarrow C^{\ad}_{g^\ad,0}$ is also a finite \'{e}tale map.

Since we want to be able to consider Hecke correspondences which differ by an element of the center, observe that we have the following:
\begin{proposition}
\label{prop:isomcentral}
    Let $z\in Z_G(\bbQ_p)$, and let $\mathscr{C}_{z}$ denote the integral model of the Hecke correspondence corresponding to $1_{zK_p}$. Then the projection maps $p_1,p_2:\mathscr{C}_{z}\rightarrow \sS_{K_p}(G,X)$ are isomorphisms.
\end{proposition}
\begin{proof}
    We first observe that these projection maps $p_1,p_2$ are isomorphisms on the generic fiber. Observe that we can check this after base change to $\bbC$, and if we let $x=(g,h)\in \Sh_{K_p}(G,X)(\bbC)$ for $g\in G(\bbA_f)$ and $h\in X$, then observe that if the pair $(x',x)\in C_z$, then we must have $x'=(h,zg)\in \Sh_K(G,X)(\bbC)$. In particular, we see that `multiplication-by-$z$' defines an isomorphism $f_z:\Sh_{K_p}(G,X)\rightarrow \Sh_{K_p}(G,X)$, and we can similarly define an isomorphism $f_{z^{-1}}$. By the extension property of canonical integral models, these extend to maps $f_z,f_{z^{-1}}:\sS_{K_p}(G,X)\rightarrow \sS_{K_p}(G,X)$ which are inverses to each other, hence we see that $\mathscr{C}_z$ is simply the graph of $f_z$, which we have shown to be an isomorphism. 
\end{proof}

\subsection{Mod $p$ Correspondences for abelian type}
We will now analyze the mod $p$ Hecke correspondences, using results from the correspondences for $p-\Isog_{G_1}$ for $(G_1,X_1)$. We now let $C_0$ denote any algebraic correspondence on $\barSs_{K_p}(G,X)$ which arises in the special fiber of some $1_{K_pgK_p}$.

Firstly, we can still define a Newton stratification on any such mod $p$ correspondence $C_0$. For any $[b]\in B(G,\upsilon)$, it is the locally closed $C^{[b]}_0$ of $C_0$ defined by the preimage via $p_1$ of $\barSs_{K_p}(G,X)^{[b]}$. We first make the following observation, which implies that we could also have taken the preimages via $p_2$:
\begin{proposition}
    \label{prop:newtonstrataHecke}
    The image of $C_0^{[b]}$ under $p_2$ also lies in $\barSs_{K_p}(G,X)^{[b]}$.
\end{proposition}
\begin{proof}
    Suppose that we have a pair $(x,x')\in C_0^{[b]}(\bar{\bbF}_p)$. By construction, such a pair lifts to an $O_K$-point $(\tilde{x},\tilde{x}')$, where $K$ is some finite extension of $L$, where $\tilde{x},\tilde{x}'$ are quasi-isogenous via $g\in G(\bbQ_p)$. We can view $\tilde{x}$ as a point on the adic generic fiber $\calS_{\phi_0}$, as defined in \ref{subsub:RZadic}, and we see that the compatibility with Hecke actions on the generic fiber implies that $\tilde{x}'$ also determines a point on $\calS_{\phi_0}$, and hence its mod $p$ reduction lies in $\barSs_{K_p}(G,X)^{[b]}$.
\end{proof}

Following the discussion of Newton strata under the map to adjoint Shimura varieties in \S\ref{subsec:Newtonstrata} shows that $C^{\ad,\ord}_0$ is the image of the projection to $\bar{\calS}_K(G^{\ad},X^{\ad})^{\ord}\times \bar{\calS}_K(G^{\ad},X^{\ad})^{\ord}$ of $C^{\ord}_0$, and also $C^{\ord}_{1,0}$, assuming that $g^\ad$ lifts to an element $g_1\in G_1(\mathbb{Q}_p)$.

\subsubsection{}\label{liftingcochar}We now want to recall some facts about lifting cocharacters from $G$ to $G_1$. Let $\tilde{G}$ be the simply connected cover of $G_1^{\der}$. Let $Z_G$ denote the center of the group $G$. Recall that we have a central isogeny
\begin{equation*}
    Z\times \tilde{G}\rightarrow G,
\end{equation*}
and thus, for a maximal torus $T$ of $G$ defined over $\bbQ$, we have a injective map with finite cokernel
\begin{equation*}
    X_*(Z_G)\oplus X_*(T^{\der})\hookrightarrow X_*(T)
\end{equation*}
where $T^{\der}$ in $\tilde{G}$ is a maximal torus.

In particular, observe that for any cocharacter $\lambda\in X_*(T)$, there exists some positive integer $m$ such that $\lambda^m$ lifts (up to some cocharacter in $X_*(Z_G)$) to a cocharacter of $\tilde{G}$, and hence to $G_1^\der$.

\begin{proposition}
\label{prop:dense}
For any $g\in G(\bbQ_p)$, consider the correspondence $C$ associated with $1_{K_pgK_p}$ as above. Then we have $C_0$ has a dense ordinary locus. Moreover, the projection map $p_1,p_2$ are finite.
\end{proposition}
\begin{proof}
Note that if there exits some $g_1$ such that $g_1^\ad=g^\ad$, then $C_{0,1}$ has a dense ordinary locus by an equivalent result for $p-\Isog(G_1,X_1)$ Proposition \ref{prop:pisogdense} quoted below. In particular, since $C^\ad_{0}$ is the image of $C_{1,0}$ under a finite \'{e}tale map, and from \S\ref{subsec:Newtonstrata} we know that the Newton strata under this map is compatible, we know that the Hecke correspondence $C^{\ad}_0$ has a dense ordinary locus. 

Thus, since the image of $C_0$ under a finite map to $\barSs_{K_p^\ad}(G^{\ad},X^{\ad})\times \barSs_{K_p^\ad}(G^{\ad},X^{\ad})$ has a dense ordinary locus, the original Hecke correspondence $C_0$ must have had a dense ordinary locus as well. 

Otherwise, we use the Cartan decomposition to replace $g$ with some $\lambda(p)$ for a dominant cocharacter $\lambda$, and consider some $g'=\lambda^n(p)\lambda'(p)$, such that $\lambda'$ is central, and we know from above that this lifts to $G_1$. Observe that from Proposition \ref{prop:isomcentral} the correspondences $C_{g',0}$ and $C_{\lambda^n(p),0}$ are isomorphic. Thus, it now suffices to show that if  $g=\lambda(p)$ does not have a dense ordinary locus, neither does the mod $p$ fiber of the correspondence associated with the double coset $1_{K_p\lambda^n(p)K_p}$. To see this, we suppose that there is some irreducible component of $C_{g,0}$ which does not intersect the ordinary locus. Let $\bar{\eta}$ be the geometric generic point of $C_{g,0}$, and observe that we can lift $\bar{\eta}$ over some local ring $R$ with residue field $\bar{\eta}$. Over $F=\mathrm{Frac}R$ we hence have a pair $(x',x)$, or more precisely there is some lift of $x$ to $y\in \Sh(G,X)$ such that $\pi_{K_p}(gy)=x'$. We can hence consider the pair $(x'=\pi_{K_p}(gy),\pi_{K_p}((g^{-1})^{n-1}y))$ which is a point on the generic fiber of the correspondence associated with the double coset $1_{K_p\lambda^n(p)K_p}$, and if we look at the mod $p$ reduction of the Zariski closure of this in $\sS_{K_p}(G,X)\times \sS_{K_p}(G,X)$, we clearly have some irreducible component whose projection via $p_1$ lies outside the ordinary locus as well, and hence this component also does not have a dense ordinary locus. 

Now, to show that $p_1,p_2$ are finite, first observe that if suffices to show that the map $p_1$ is finite, since otherwise we can look at the dual correspondence $C_0^\vee$, given as the space of pairs $(y,x)$ for $(x,y)\in C_0$. Note that if $C_0$ appears as a component of $C_{g,0}$, then such a $C_0^\vee$ will appear as a component of $C_{g^{-1},0}$. If the map $p_1$ is not finite, then we see that this implies that $p_1(C_0)$ is a proper subset of $\barSs_{K_p}(G,X)$ of strictly smaller dimension. 

The same analysis as above allows us to see that if $g^\ad$ admits a lifting to $g_1$, then we will have some component $C_{1,0}$ which under the projection map $p_1$ has an image of strictly smaller dimension, and hence the same must be true of $p-\Isog_{G_1}$, which is not possible. 

Otherwise, we can again replace $g$ with $\lambda(p)$, and we see that we must have a component of $C_{\lambda^n(p),0}$ with this property that the image under $p_1$ is of strictly smaller dimension, which again is not possible, again from applying Proposition \ref{prop:pisogdense} below.
\end{proof}

\begin{proposition}
\label{prop:pisogdense}
For any correspondence $C_{1,0}$ appearing in the mod $p$ fiber of some Hecke correspondence of $\barSs_{K_1}(G_1,X_1)$, the ordinary locus $C_{1,0}^\ord$ is dense in $C_{1,0}$. Moreover, we have that the projection maps $p_1,p_2:C_{1,0}^\ord\rightarrow \barSs_{K_{1p}}(G_1,X_1)^\ord$ are finite.
\end{proposition}
\begin{proof}
    From \ref{subsub:comparepIsog}, we see that it suffices to consider the corresponding statement for $p-\Isog_{G_1}\otimes\kappa$, the mod $p$ fiber of $p-\Isog$. This follows from \cite[Remark 6.1.22]{Lee20}, or more precisely the observation there is that the ordinary density condition holds when the only unramified element in $B(G_1,\{\upsilon_1\})$ is ordinary. Moreover, for a split group $G$, only the unramified element in $B(G,\{\upsilon\})$ is ordinary. In paticular, since $[b]$ is unramified if and only if $[b^\ad]$ is unramified, and we have bijections $B(G,\{\upsilon\})\simeq B(G^\ad,\{\upsilon^\ad\})\simeq B(G_1,\{\upsilon_1\})$, we see that only the ordinary $\sigma$-conjugacy class $[b_1^\ord]$ in $B(G_1,\{\upsilon_1\})$ is unramified, and thus the density statement follows. It remains to see that the projection maps $s,t:p-\Isog_{G_1}\otimes\kappa\rightarrow \barSs_K(G_1,X_1)$ are finite when restricted to the ordinary locus, but this follows from the fact that the fibers of $s,t$ are a subvariety of $X_{\upsilon_1}^{G_1}(b^\ord_1)$ for some choice of $b^\ord_1$ in the ordinary $\sigma$-conjugacy class, and $\dim X^{G_1}_{\upsilon_1}(b^\ord_1)=0$
\end{proof}

\subsection{Action of correspondences on points}
We briefly recall what it means for a Hecke correspondence to act on a geometric point of the Shimura variety. Let $C\in \Cor(\Sh_{K_p}(G,X),\Sh_{K_p}(G,X))$, such that the projection maps are generically finite and let $x\in \Sh_{K_p}(G,X)(\bbC)$. We can view $[x]\in \mathrm{Ch}_0(\Sh_{K_p}(G,X))$, the Chow group of zero-dimensional cycles on $\Sh_{K_p}(G,X)$, and we can let $C$ act on $[x]$ as defined in \cite[Definition A.4]{B2002}, which we slightly modify to have a left action instead. Let us briefly recall how this goes. We have a refined Gysin map
\[i^!:\mathrm{Ch}_d(C\times_{\bbF_p} \barSs_{K_p}(G,X))\rightarrow \mathrm{Ch}_0(C\times_{\barSs_{K_p}(G,X)}\barSs_{K_p}(G,X))\]
and a pushforward map
\[p_{1*}:\mathrm{Ch}_0(C)\rightarrow \mathrm{Ch}_0(\barSs_{K_p}(G,X)),\]
and we define $C\cdot x=p_{1*}i^!([C\times\{x\}])$.

In particular, from the definition we see that if the correspondence $C$ has projections $p_1,p_2$ finite, then since the intersection of $p_2^{-1}(x)$ and $C$ is proper, the resultant cycle $i^!([C\times\{x\}])$ in $\mathrm{Ch}_0(C)$ is simply $p_2^{-1}(x)$, thus $C\cdot x$ is simply given by a finite sum
\[C\cdot x=\sum_{i}a_i[x_i],\]
where $x_j$'s are elements of the finite set $p_{1}(p_2^{-1}(x))$, and $a_i$ is the multiplicity of the fiber $p_1^{-1}(x_i)$.

Specializing this to the case of Hecke correspondences give us the following lemma:
\begin{lemma}
\label{lemma:genericHecke}
    If $C=C_g$ for a double coset $1_{K_pgK_p}$, and we can write $K_pgK_p$ in terms of left coset decomposition as a finite disjoint union
    \[K_pgK_p=\bigsqcup_{j} b_jg_jK_p,\]
    where $b_j\in\mathbb{Z}$ and $g_j\in G(\bbQ_p)$, then for any $x\in \Sh_{K_p}(G,X)$ and any preimage $\tilde{x}$ of $x$ under the projection map $\pi_{K_p}:\Sh(G,X)\rightarrow \Sh_{K_p}(G,X)$, we have
    \[C\cdot x=\sum_j b_j [\pi_{K_p}(g_j\tilde{x})]\]
\end{lemma}
We also may define an action of a mod $p$ Hecke correspondence, $C_0\in \Cor(\barSs_{K_p}(G,X),\barSs_{K_p}(G,X))$, and let $x\in \barSs_{K_p}(G,X)(\bar{\bbF}_p)$. We can again view $[x]\in \mathrm{Ch}_0(\barSs_{K_p}(G,X))$ as a zero-dimensional cycle on $\barSs_{K_p}(G,X)$, and the definition \cite[Definition A.4]{B2002} works as well in this situation to define an action of correspondences on points. Now, we may compare the actions on the generic fiber and special fiber:
\begin{proposition}
\label{prop:comparemodp}
    Suppose that we have $C,D\in \Cor(\Sh_{K_p}(G,X),\Sh_{K_p}(G,X))$, and let $C_0,D_0$ be the mod $p$ reduction. Let $x_0\in \barSs_{K_p}(G,X)^\ord(\bar{\bbF}_p)$ be an ordinary point, and suppose that $x\in \Sh_{K_p}(G,X)(\bbC)$ is a lift of $x_0$. Then we have the equalities
    \[\overline{C\cdot D}=C_0\cdot D_0\]
    \[\overline{C\cdot x}=C_0\cdot x_0,\]
    where $\overline{(-)}$ means taking the mod $p$ reduction.
\end{proposition}
\begin{proof}
    This follows from \cite[Lemma B.1]{B2002}. More precisely, if we consider the open subscheme $\sS'=\sS_K(G,X)\backslash\barSs_{K_p}(G,X)^{\mathrm{non}-\ord}$, given by removing the non-ordinary part of the special fiber, the restriction of the Zariski closure $\mathscr{C}\in\Cor(\sS_{K_p}(G,X),\sS_{K_p}(G,X)$ of $C$ (similarly $\mathscr{D}$ of $D$) to this open satisfies the condition that $p_1,p_2$ is finite using Proposition \ref{prop:dense}. We can then take the Zariski closure of $C\cdot x$ in $\sS'$, or $C\cdot D\in \sS'\times \sS'$ and we are in a situation to apply \cite[Lemma B.1]{B2002}. 
\end{proof}
Finally, note that the same argument as in \cite[Lemma A.3]{B2002} implies the following associativity result:
\begin{lemma}
    Let $C,D$ be correspondences in $\Cor(\Sh_{K_p}(G,X),\Sh_{K_p}(G,X))$ such that the projection maps $p_1,p_2$ are generically finite when restricted to $C,D$. Then 
    \[C\cdot (D\cdot x)=(C\cdot D)\cdot x\]
\end{lemma}
\subsection{Canonical liftings of ordinary points}
In the following, we will fix the element $b=\upsilon(p)$ as a representative of the ordinary $\sigma$-conjugacy class, and so the uniformization maps $\Theta_{x,b}$ will be for this choice of $b$, hence we will suppress the $b$. 
\begin{proposition}
\label{prop:ordlift}
    Let $x$ be an ordinary point in $\barSs_{K_p}(G,X)^\ord(\bar{\bbF}_p)$. Then, there exists some special point $\tilde{x}$ given by $({\bf T},h)$ which satisfies $\mathbf{T}_{\bbQ_p}\simeq T$, and under this isomorphism we have $\mu_{h,\bbQ_p}=\mu$.
\end{proposition}
\begin{proof}
    Note that there exists some $g^p$ such that the point $x$ lies in $\barSs_{K_{p}}(G, X)^+$, since by \cite[2.2.5]{K2010} $G(\bbA_f^p)$ acts transitively on connected components of $\sS_{K_{p}}(G, X)$. 

    From the discussion in \S\ref{subsec:Newtonstrata} we can find an ordinary point $x_1\in \sS_{K_{1,p}}(G_1,X_1)^+(\overline{\bbF}_p)$ which maps to $x$. We know from \cite[Theorem 3.5]{SZ2016} that for every ordinary point in $\sS_{K_{1,p}}(G_1,X_1)^+(\overline{\bbF}_p)$, there exists a special point lifting $\tilde{x}_1$, with associated Shimura datum $(\mathbf{T}_1,h_1)$ such that $\mathbf{T}_{1,\bbQ_p}$ is conjugate to $T_1$ and under this association we have $\mu_{h_1,\bbQ_p}=\mu_1$

    Consider now the image $\tilde{x}$ of $\tilde{x}_1$ under the map $\sS_{K_{1p}}(G_1,X_1)^+\rightarrow \sS_{K_{p}}(G,X)^+$, then $\tilde{x}$ is a special point. Denote the special point datum associated with $\tilde{x}$ by $(\mathbf{T},h)$. Observe that we have $\mathbf{T}^\ad_{\bbQ_p}$ is conjugate to $T^\ad$, so $\mathbf{T}_{\bbQ_p}$ is quasi-split and hence up to conjugating $\mathbf{T}$ we may assume $\mathbf{T}_{\bbQ_p}=T$. Now, we want to show that the cocharacter $\mu_h$ satisfies $\mu_{h,\bbQ_p}=\mu$. We first observe that from \ref{liftingcochar} that any cocharacter $\lambda$, is determined by the map to the adjoint group $\lambda^\ad$, since for any cocharacter $\lambda:\mathbb{G}_m\rightarrow G$, it is determined by the induced maps to $G/G^{\der}$ and $G^{\ad}$. Observe that since $G/G^{\der}$ is commutative, the quotient map $\mu_h:\mathbb{G}_m\rightarrow G\rightarrow G/G^{\der}$ is constant for all elements $h\in X^+$. Thus, we see that the associated cocharacter $\mu_{h}$ satisfies $\mu_{h,\bbQ_p}=\mu$, since $\mu_h^{\ad}=\mu_{h_1}^{\ad}=\mu^\ad$.
\end{proof}
\subsubsection{} Let $M:=M_\mu$ be the Levi centralizer of $\mu$. In fact, this special point lift $\tilde{x}$ is an $M$-adapted lift, a notion we now define. As in the proof of Proposition \ref{prop:newtonstrataHecke} any $O_K$ lift $\tilde{x}$ of $x$, where $K$ is some finite extension of $L$, defines a point on the adic space $\calS_{\phi_0}$, and from \ref{eqn:uniformization} we have an isomorphism 
\[\calS_{\phi_0}\simeq\bigsqcup_{\phi,\phi^\ad=\phi_0}I_\phi(\bbQ)\backslash \mathcal{M}\times G(\mathbb{A}^p_f).\]
Let $\mathcal{N}\subset\mathcal{M}$ be the closed adic space corresponding to the closed formal subscheme
\[\RZ(M,b,\{\upsilon\})\subset \RZ(G,b,\{\upsilon\}),\]
where recall we require that the representative $b$ lies in $M(\bbQ_p)$, which holds since we took $b=\upsilon(p)$. We say that $\tilde{x}$ is $M$-adapted if under the isomorphism above, we have 
\[\tilde{x}\in \bigsqcup_{\phi,\phi^\ad=\phi_0}I_\phi(\bbQ)\backslash \mathcal{N}\times G(\mathbb{A}^p_f).\]
\begin{remark}
    Another possible equivalent way to define this would be to look at the complete local ring at $x$, and define a closed formal subscheme corrsponding to lifts with $M$-structure, by identifying the complete local rings for $\sS_K(G,X)$ and $\sS_{K_1}(G,X)$, and applying the definition for $M_{\mu_1}$-adapted liftings for $p$-divisible groups as was done in \cite[\S3.1]{Lee20}. 
\end{remark}
\begin{proposition}
    The lift $\tilde{x}$ constructed is Proposition \ref{prop:ordlift} is $M$-adapted.
\end{proposition}
\begin{proof}
    Firstly, note that the special point lifting $\tilde{x}_1$ constructed in \cite[Theorem 3.5]{SZ2016} is shown to be $M_{\mu_1}$-adapted. Now, it remains to check that if $\tilde{x}_1\in I_{\phi_1}\backslash(\RZ(M_1,b_1,\{\upsilon_1\})^{+\ad} \times G(\bbA_f^p))$, where $\RZ(M_1,b_1,\{\upsilon_1\})^{+,\ad}$ denotes the adic generic fiber of the connected component of $\RZ(M_1,b_1,\{\upsilon_1\})^{+}$ containing $1$, then the image $\tilde{x}$ must lie in $\bigsqcup_{\phi,\phi^\ad=\phi_0}I_\phi(\bbQ)\backslash \mathcal{N}\times G(\mathbb{A}^p_f)$. But this follows from the fact that under the isomorphism
    \[\mathcal{M}\simeq (J_{b}(\bbQ_p)\times \RZ(G_1^\ad,b_1^\ad,\{\upsilon_1^\ad\})^{+,\ad})/J_b(\bbQ_p)^+,\]
    we have $\mathcal{N}\simeq (J_{b}(\bbQ_p)\times \RZ(M_1^\ad,b_1^\ad,\{\upsilon_1^\ad\})^{+,\ad})/J_b(\bbQ_p)^+$.
\end{proof}
\subsubsection{}
\label{subsub:speciallift}
We now consider $\tilde{x}$ as in the previous proposition. Let $g\in G(\bbA_f)$ be such that $\tilde{x}$ lies in the image of $g\cdot \Sh(\mathbf{T},h)$. Then, we may choose a preimage $\tilde{x}'$ of $\tilde{x}$ under the projection map 
\[\pi_{K_p}:\Sh(G,X)\rightarrow \Sh_{K_p}(G,X)\]
which also lies in the image of $g\cdot \Sh(\mathbf{T},h)$. We want to relate the Frobenius mod $p$ with the action of $j=\mu_h^{-1}(p)$.

\begin{proposition}
\label{prop:Frob}
Let $x$ be an ordinary point in $\sS_{K_{p}}(G,X)(\overline{\bbF}_p)$, and $\tilde{x}'$ the special point lift constructed in \ref{subsub:speciallift}. Let $j=\mu_h^{-1}(p)$. Then, the mod $p$ reduction of $\pi_{K_p}(j\tilde{x})\in \Sh_{K_p}(G,X)$ is exactly the image of the relative Frobenius of $\barSs_{K_p}(G,X)$ acting on $x$.
\end{proposition}
\begin{proof}
    This follows from the theory of special points. More precisely, observe that the condition that $\mu_{h,\bbQ_p}=\mu$ implies that $\tilde{x}$ is defined over $\bbQ_p$. By construction, we have that the action of the geometric Frobenius $\sigma\in\mathrm{Gal}(\bar{\bbQ}_p/\bbQ_p)$ on $\tilde{x}$ acts by left multiplication by $\mu_h^{-1}(p)$. Thus, we see that the mod $p$ reduction of $\pi_{K_p}(j\tilde{x})$ is the image of $x$ under the relative Frobenius.
\end{proof}

We now want to show the following proposition, which is a generalization of \cite[Lemma 4.5]{B2002}, which determines the action mod $p$ of $u\in U(\bbQ_p)$ on such canonical liftings.

\begin{proposition}
Let $x$ be an ordinary point in $\sS_{K_{p}}(G,X)(\overline{\bbF}_p)$, and let $\tilde{x}'\in\Sh(G,X)$ the special point lift constructed in \ref{subsub:speciallift}. Let $U$ denote the unipotent radical of the parabolic subgroup $P$ of $G$ associated to $\mu$. Then for any $u\in U(\bbQ_p)$, we have that the mod $p$ reduction of $\pi_{K_p}(u\tilde{x}')\in\mathscr{S}_{K_p}(G,X)$ is just $x$.
\label{prop1}
\end{proposition}
\begin{proof}
We first show this for a similarly defined special-point lift $\tilde{x}'_1\in \Sh(G_1,X_1)$ for $x_1$ an ordinary point on the Hodge type cover. To see this, note that from \cite[1.2.18]{K2017} we have a map $G_1(\bbQ_p)/G_1(\bbZ_p)\rightarrow G_1(L)/G_1(O_L)$ that sends $u$ to the image of the reduction mod $p$ of the isogeny given by $u$ on $\tilde{x}_1$. Let $g_0\in X_{\upsilon_1}^{G_1}(b^\ord_1)$ be the image, where here we choose $b_1^\ord=\upsilon_1(p)$, since by \cite[\S7]{Wor2013} we can always choose a trivialization of the $p$-divisible group $\mathscr{G}_{x_1}$ over $x_1$ such that the Frobenius on $\bbD(\mathscr{G}_{x_1})$ is $b_1^\ord\sigma$. We want to say that this element $g_0$ is simply $1$, and to see this we can apply the result of \cite[Prop 2.5.4]{Lee20}, to see that since $\tilde{x}_1$ is $M_1$-adapted, where $M_1=M_{1\mu_1}$, we must have $g_0\in U(L)$. 

Moreover, observe that with this choice of representative $b_1^\ord=\upsilon(p)$, we can identify $X_{\upsilon_1}^{G_1}(b^\ord_1)$ with the discrete set $M_{1}(\bbQ_p)/M_{1}(\bbZ_p)$, from \cite[Proposition~2.5.4]{CKV}. More precisely, we can see this by using the Iwasawa decomposition to write any element $g_0\in X_{\upsilon_1}^{G_1}(b^\ord_1)$ as $g\cdot G_1(O_L)$ for $g\in P(L)$, and sending $g=um$ for $u\in U(L),m\in M(L)$ to $mM_1(O_L)$. Then \cite[Proposition~2.5.4]{CKV} shows that this induces a bijection from $X_{\upsilon_1}^{G_1}(\upsilon(p))$ to $X_{\upsilon_1}^{M_1}(\upsilon(p))$, and since $\upsilon$ is central in $M_1$ we have $X_{\upsilon_1}^{M_1}(\upsilon(p))\simeq M_1(\bbQ_p)/M_1(\bbZ_p)$. In particular, the condition that $g_0\in U(L)$ implies that $g_0$ is 1 under all these identifications.

Up to the action of some $g^p\in G(\bbA_f^p)$, we may assume that $x\in \sS_{K_{p}}(G,X)^+(\overline{\bbF}_p)$. Observe that we have an isomorphism of root systems $\Phi(G,T)=\Phi(G^{\ad},T^{\ad})=\Phi(G_1,T_1)$. Hence if we consider $U_1$ the unipotent radical of the standard parabolic subgroup of $G_1$ corresponding to $\mu_1$, then we can identify $U_1$ with $U$. In particular, since this result is true for the lift $\tilde{x}_1$, for any $u\in U_1(\bbQ_p)$, the same is true for the action of $u\in U(\bbQ_p)$ on $\tilde{x}$. Finally, we see that the action of $G(\bbA^p_f)$ commutes with the action of $G(\bbQ_p)$, and with mod $p$ reduction, hence the result we want is also true for any $x$.
\end{proof}
We can also precisely determine the mod $p$ reduction of $\pi_{K_p}(\mu_i^{-1}(p)\tilde{x}')$. Recall that we have a map $\Theta_{x}:=\Theta_{x,b=\upsilon(p)}$ constructed in the previous section. Note that by \cite[\S7]{Wor2013} we may also arrange the isomorphisms of connected components such that $1\in X_{\upsilon_1}^{G_1}(\upsilon_1(p))$ goes to $1\in X_\upsilon^G(b)$, and hence we assume that $\Theta_{x,b}(1)=x$. Moreover, note that with this choice of $b$, given an element $z\in Z_M(\bbQ_p)$ where $Z_M$ is the center of $M$, we have $z\in X_\upsilon^G(b)$, and hence from the remark following Theorem \ref{thm:RZconstruct} we view $z\in \RZ(G,b,\{\upsilon\})$.
\begin{proposition}
\label{prop:frob_i}
Let $x$ be an ordinary point in $\sS_{K_{p}}(G,X)(\overline{\bbF}_p)$, and let $\tilde{x}'\in\Sh(G,X)$ the special point lift constructed in \ref{subsub:speciallift}. Then, for any $\lambda\in X_*(T)$ such that $\lambda(p)$ lies in $Z_M(\bbQ_p)$ where $Z_M$ denotes the center of $M$, we have the mod $p$ reduction of $\pi_{K_p}(\lambda(p)\tilde{x}')$ is $\Theta_x(\lambda(p))$.
\end{proposition}
\begin{proof}
We can see this by analyzing the Hecke action on the tower of adic spaces $(\mathcal{N}_{K})_K$ for $K\subset M(\bbQ_p)$ a compact open subgroup. More precisely, we see that the Hecke actions of $M(\bbQ_p)$ on $(\mathcal{N}_{K})_K$ is simply the restriction of the $G(\bbQ_p)$ action on the tower $(\mathcal{M})_{K'_p}$, and since the lift $\tilde{x}$ is $M$-adapted, and $\lambda(p)\in T(\bbQ_p)\subset M(\bbQ_p)$, it suffices to study the Hecke action on $(\mathcal{N}_{K})_K$. Now, since $\lambda(p)$ is central in $M$, we have that this Hecke action agrees with the usual action on the formal scheme $\RZ(M,b,\{\upsilon\})$ where we view $Z_M(\bbQ_p)$ as a subset of $J_b(\bbQ_p)$, as explained in \ref{subsub:actionZG}. In particular, the mod $p$ reduction of this action just sends, when looking at the affine Deligne-Lusztig variety, an element $g\in X_\upsilon^G(b)$ to $\lambda(p)g$. Thus, we see that for any $\tilde{x}$ which is $M$-adapted, the mod $p$ reduction of $\pi_{K_p}(\lambda(p)\cdot \tilde{x}')$ is just $\Theta_{x}(\lambda(p))$, as desired.
\end{proof}

\subsection{Partial Frobenius for abelian type}
We would now like to define the partial Frobenius to be the $p$-power quasi-isogeny represented by $\mu_i^{-1}(p)$. Since we are in the abelian type case, we cannot work directly with $p$-divisible groups. Instead, here we will define the partial Frobenius correspondence, at least over a dense set  of points of the ordinary locus, and take the Zariski closure. Again, we continue to fix the element $b=\upsilon(p)$ as a representative of the ordinary $\sigma$-conjugacy class, and thus we see that $\mu_i^{-1}(p)\in X_\upsilon^G(b)$.

Define $\Frob_{i}$ as the Zariski closure in $\barSs_{K_p}(G,X)\times \barSs_{K_p}(G,X)$ of all the pairs $(\Theta_x(\mu^{-1}_i(p)),x)$ for all $x\in \barSs_{K_p}(G,X)^\ord$. Since $\Theta_{g^px}(\mu^{-1}_i(p))=g^p\Theta_x(\mu^{-1}_i(p))$ for all $g\in G(\bbA_f^p)$, we see that this Zariski closure is the union of some closed irreducible subschemes of $\barSs_{K_p}(G,X)\times \barSs_{K_p}(G,X)$ of dimension $d$, where $d=\dim \barSs_{K_p}(G,X)$.We have the following.
\begin{proposition}
    \label{prop:Frob_i}
    The closed points of the restriction $\Frob_{i}^\ord$ of $\Frob_{i}$ to the ordinary locus consists exactly of pairs $(\Theta_x(\mu^{-1}_i(p)),x)$.
\end{proposition}
\begin{proof}
    $\Frob_i^\ord$ clearly contains all such points, and we want to show that we have no other closed points. We first observe that by generic flatness, there is some open dense $U\subset \barSs_{K_p}(G,X)$ such that the map $p_2^{-1}(U)\rightarrow U$ is flat, and the fibers over $x\in U$ (ignoring multiplicity) are all a single point. Let $x\in\barSs_K(G,X)^\ord$. We want to show that translating by $g\in G(\bbA_f^p)$ will allow us to translate any small enough open neighborhood of $x$ into $U$. Observe that this would follow if we know that the Zariski closure of the $G(\bbA_f^p)$-orbit of $x$ is dense in $\barSs_K(G,X)^\ord$. We may see this by first noting that up to translating by $g\in G(\bbA_f^p)$ we may assume that $x\in \barSs_K(G,X)^{\ord,+}$, and let $x_1$ be a lift to $\barSs_{K_1}(G_1,X_1)^{\ord,+}$. Then, we see that from the ordinary Hecke orbit conjecture for Hodge type as proved by van Hoften in \cite{vH2024}, the intersection of the Zariski closure of the $G_1(\bbA_f^p)$-orbit of $x_1$ and $\barSs_{K_1}(G_1,X_1)^{\ord,+}$ is the whole $\barSs_{K_1}(G_1,X_1)^{\ord,+}$. This implies that the intersection of the Zariski closure of the $G^\ad(\bbA_f^p)$-orbit of $x^\ad\in \barSs_{K^\ad}(G^\ad,X^\ad)^{\ord,+}$ and $\barSs_{K^\ad}(G^\ad,X^\ad)^{\ord,+}$ is the whole $\barSs_{K^\ad}(G^\ad,X^\ad)^{\ord,+}$ as well, since the projection map is finite. Thus, we also see that the Zariski closure of the $G(\bbA_f^p)$-orbit of $x\in \barSs_{K}(G,X)^{\ord,+}$ and $\barSs_{K}(G,X)^{\ord,+}$ is the whole $\barSs_{K}(G,X)^{\ord,+}$, so we are done.

    In particular, the relation $\Theta_{g^px}(\mu^{-1}_i(p))=g^p\Theta_x(\mu^{-1}_i(p))$ for all $g\in G(\bbA_f^p)$ means that after translating into this open $U$ the fibers of $p_2^{-1}(g^px)$ just contain a single point, so $p_2^{-1}(x)$ also contains a single point (ignoring multiplicity). In fact, the argument shows that if we look at the perfections $\barSs_{K_p}(G,X)^{\ord,\mathrm{perf}}\subset \barSs_{K_p}(G,X)^{\mathrm{perf}}$, and we see that the projection maps $\Frob_i^{\ord,\mathrm{perf}}\xrightarrow{p^{\mathrm{perf}}_2} \barSs_{K_p}(G,X)^{\ord,\mathrm{perf}}$ is an isomorphism, since it is a universal homeomorphism. A similar statement is also true for $p_1$.
\end{proof}

\subsection{Composition of algebraic correspondence}
We first recall the following result we will continually use to show vanishing of algebraic cycles in the situation where we know the projections are generically finite.
\begin{proposition}[{\cite[Lemma A.6]{B2002}}]
\label{prop:Bueltelvanishing}
Let X and Y be smooth $d$-dimensional varieties over an algebraically closed field $k$. Let $C$ be a generically finite correspondence from $X$ to $Y$. Assume that there are given open dense subvarieties $X^\circ \subset X$ and $Y^\circ \subset Y$ , such that the restriction $C^\circ$ of $C$ to $X^\circ\times Y^\circ$ has projection maps $p_1,p_2$ finite. Assume that $C\cdot x$ vanishes for all closed points $x \in X^\circ$. Then $C$ vanishes.
\end{proposition}
We can apply the above to $X=Y=\barSs_{K}(G,X)$ and $X^\circ=Y^\circ=\barSs_{K}(G,X)^\ord$, and hence to show equality of correspondences we only need to show equality for the action on ordinary points. Moreover, it is easy to see that the action of $\Frob_i$ commutes with the action of $G(\mathbb{A}_f^p)$, since this holds over the ordinary locus by construction. We first show the following:
\begin{lemma}
    For each $i$, the projection map $p_2$ is \'{e}tale, and $p_1:\Frob^\circ_i\rightarrow \sS_{K_p}(G,X)$ is purely inseparable of degree $p^{\langle2\rho,\mu_i\rangle}$.
\end{lemma}
\begin{proof}
    Observe that as shown in the proof of Propositon \ref{prop:Frob_i} the projection maps are bijections on geometric points, so the inseparable degree is the same as the degree of the respective maps. We first project $\Frob_i$ to a correspondence $\Frob_i^\ad$ on $\bar{\calS}_{K^\ad_p}(G^\ad,X^\ad)$, and observe that since the map $\Frob_i\rightarrow \Frob^\ad_{_i}$ is finite \'{e}tale, we only need to check the inseparable degrees of the projection maps of $\Frob^\ad_i$. More precisely, we have a commutative square
    \begin{equation*}
    \begin{tikzcd}
    \Frob_i \arrow[r] \arrow[d,"p_1"] & \Frob^\ad_i \arrow[d,"p^\ad_1"] \\
    \barSs_{K}(G,X) \arrow[r]& \barSs_{K^\ad}(G^\ad,X^\ad),             
\end{tikzcd}
    \end{equation*} 
    where the horizontal maps are all finite \'{e}tale, thus we see that the inseparable degree of $p_1$ is the same as $p_1^\ad$, and similarly for $p_2$.

    Now, observe that from \ref{liftingcochar} there exists some positive integer $m$ such that $(\mu_i^\ad)^{-m}$ can be lifted to a cocharacter $\lambda_i$ in $X_*(T_1)$, up to a central cocharacter. Note that this central cocharacter does not change the inseparable degrees, from Proposition \ref{prop:isomcentral}. 
    
    The same argument as above shows that we only need to determine the inseparable degrees of an algebraic correspondence $D_{\lambda_i}$ on $\barSs_{K_1}(G_1,X_1)$ which we may define similarly to the definition of $\Frob_i$, that is, we look at the Zariski closure of all pairs $(\Theta_x(\lambda_i),x)$ for $x\in \barSs_{K_1}(G_1,X_1)^\ord$. In particular, we want to show that for $D_{\lambda_i(p)}$, we have $p_2$ is \'{e}tale, while $p_1$ is inseparable of degree $p^{\langle 2\rho,-\lambda_i\rangle}$. Indeed, if we let $p^{a}$ be the inseparable degree of $p_2$ for $\Frob_i$, and $p^{b}$ be the inseparable degree of $p_1$ for $\Frob_i$, then we see that the inseparable degree of $p_2$ for $D_{\lambda_i(p)}$ is $p^{ma}$ (resp. for $p_1$ is $p^{mb}$), and so we deduce $a=0$, $b=\langle 2\rho,\mu_i\rangle$.
    
    Using \ref{subsub:comparepIsog}, we may first compare it with the following algebraic cycle in $p-\Isog_{G_1}\otimes\mathbb{F}_p$. Following \cite[6.1.14]{Lee20}, there is an algebra homomorphism 
    \[\bar{h}:\mathcal{H}(M_1(\bbQ_p)//M_1(\bbZ_p),\bbQ)\rightarrow \bbQ[p-\Isog_{G_1}\otimes\mathbb{F}_p],\] 
    where $\bbQ[p-\Isog_{G_1}^\ord\otimes\mathbb{F}_p]$ is the Chow group of top-dimensional cycles in $p-\Isog_{G_1}\otimes\mathbb{F}_p$. In particular, we can consider the cycle $\bar{h}(1_{\lambda_i(p)M_1(\mathbb{Z}_p)})$. By definition, points on $\bar{h}(1_{\lambda_i(p)M_1(\mathbb{Z}_p)})$ correspond to quasi-isogenies of points on $\barSs_{K_1}(G_1,X_1)^\ord$ which are given, on the canonical lift, by $\lambda_i(p)$. 
    
     Now, we argue similarly to \cite[5.7]{W2000} to determine the degrees. Consider a $\bbZ_p$-lattice $\Lambda$ with $M_{1,\bbZ_p}$-structure: more precisely we can decompose $\Lambda=\Lambda_0\oplus \Lambda_1$, where $\Lambda_0$ is the eigenspace where $\mu_1(z)$ acts trivially, and $\Lambda_1$ is the eigenspace where $\mu_1(z)$ acts as $z^{-1}$. Now, we can use Serre-Tate coordinates on the complete local ring at an ordinary point to determine the degrees of the maps. This complete local ring is a formal torus with character group $N$ which is a sublattice of $\mathrm{Hom}(\Lambda_0,\Lambda_1)$. Write $\lambda_i(p)=(m_0,m_1)$, where $m_0$ stabilizes $\Lambda_0$, and $m_1$ stabilizes $\Lambda_1$. Note that the complete local ring of $D_{\lambda_i}$ over any ordinary point will correspond to the subspace $N'$ of $N^2$ consisting of maps $(f_1,f_2)$ such that $f_1\circ m_0=m_1\circ f_2$. Moreover, we see that the inseparable degree of $s,t$ is the size of the preimage under the projection from $N'$ to the first (resp. second) copy of $N$.

    We hence see that since $\lambda_i(p)$ is, up to some central element, a power of $\mu_i^{-1}(p)$, we see that $m_0$ acts trivially on $\Lambda_0$, and so the inseparable degree of $t$ is $1$. For $s$, observe that since $N$ is free of rank $\langle 2\rho,\mu\rangle$, and post-composing with $m_1$ acts as $p$ on $\langle 2\rho,\mu_i\rangle$ of the basis elements of $N$, we have this inseparable degree is $p^{\langle 2\rho,\mu_i\rangle}$.

    Now, observe that under the map
    \[(s,t):p-\Isog^\ord\otimes\kappa\rightarrow \barSs_{K_{1p}}(G_1,X_1)\times \barSs_{K_{1p}}(G_1,X_1),\]
    the image of $\bar{h}(1_{\lambda_i(p)M_1(\mathbb{Z}_p)})$ is $D_{\lambda_i}$, since the corresponding image of this isogeny is between $(\Theta_x(\lambda_i),x)$. Moreover, we note that by Grothendieck-Messing theory this map is locally an isomorphism, since deforming an isogeny is just a compatible deformation of the source and the target, hence the inseparable degree of $s$ is the same as $p_1$, and similarly the inseparable degree of $t$ is the same as $p_2$, which gives us the desired result. 
\end{proof}
Thus, we have the following proposition, which establishes the first two parts of Theorem \ref{ref:mainthms2} on the level of algebraic correspondences.
\begin{proposition}
\label{prop:pconditions}
The partial Frobenius correspondence $\Frob_{\frakp_i}:=\Frob_i$ commute for each $i\neq j$, and satisfy that \begin{equation}
    \Frob_\frakp=\prod_i \Frob_{\frakp_i},
\end{equation}
where $\Frob_\frakp$ is the dual graph of the relative Frobenius at $\frakp$ on $\bar{\calS}_{K_p}(G,X)$.
\end{proposition}
\begin{proof}
Let $x\in \barSs_{K_p}(G,X)^\ord(\bar{\bbF}_p)$, and to apply Proposition \ref{prop:Bueltelvanishing}, we first observe that by construction of $\Frob_i$ and the previous proposition
\[\Frob_i\cdot x=p^{\langle 2\rho,\mu_i\rangle}[\Theta_x(\mu_i^{-1}(p))],\]
since $p_2$ for $\Frob_i$ is \'{e}tale is totally inseparable of degree $p^{\langle 2\rho,\mu_i\rangle}$. Now, observe that if we let $x_i=\Theta_x(\mu^{-1}_i(p))$, $x_j=\Theta_x(\mu^{-1}_j(p))$, then we have 
\[\Theta_{x_i}(\mu^{-1}_j(p))=\Theta_x(\mu^{-1}_j(p)\cdot\mu^{-1}_i(p))=\Theta_{x_j}(\mu^{-1}_i(p)),\]
from the construction of $\Theta_x$, the fact that $\mu_i,\mu_j$ are central in $M_\mu$, and the comparison between different base-points \ref{lemma:Theta}. In particular, this implies that 
\[\Frob_i\cdot \Frob_j\cdot x=p^{\langle 2\rho,\mu_i+\mu_j\rangle}[\Theta_x(\mu_j(p)\mu_i(p))]=\Frob_i\cdot \Frob_j\cdot x,\]
and thus we may conclude that $\Frob_j\cdot\Frob_i-\Frob_i\cdot\Frob_j=0$. 

We now proceed to show
\begin{equation}
    \Frob_\frakp\cdot x=\left(\prod_i \Frob_{\frakp_i}\right)\cdot x,
\end{equation}
where $x$ is some ordinary point. Note that by Proposition \ref{prop:Frob}, we see that over the ordinary locus closed points in $\Frob^\circ_{\frakp}$ consist exactly of pairs $(\Theta_x(\mu^{-1}(p),x)$, while we see that 
\[\left(\prod_i \Frob_i\right)\cdot x=\prod_{i}p^{\langle 2\rho,\mu_i\rangle}\Theta_x(\mu^{-1}_1\dots\mu^{-1}_d(p))=p^{\langle 2\rho,\mu\rangle}\Theta_x(\mu(p))=\Frob\cdot x.\]
where we note that the relative Frobenius has degree $\dim\barSs_K(G,X)=p^{\langle 2\rho,\mu\rangle}$, and so we are done.
\end{proof}
\begin{remark}
Since we do not have a moduli interpretation for general abelian type Shimura varieties, outside of the $\mu$-ordinary locus, it is not clear what $\Frob_{\frakp_i}$ has to do with partial Frobenii. However, the above result at least suggests that the Frobenius isogeny, viewed as an algebraic correspondence, admits a decomposition into factors for each prime $\frakp$ above $p$. 
\end{remark}
\subsection{Abstract Eichler-Shimura Relations}
\label{subsect:Heckealg}
We have isomorphisms of Hecke algebras
\begin{equation*}
    \calH(G(\bbQ_p)//K_p,\bbQ)\simeq \bigotimes_{i}\calH(G_i(\bbQ_p)//K_{i,p},\bbQ).
\end{equation*}
Similarly, since $G_{\bbQ_p}$ admits a decomposition, we can write
\begin{equation}
    \mu=\prod_{i}\mu_i
\end{equation}
where $\mu_i$ is a minuscule cocharacter of $G_i$. If we let $M$ be the centralizer of $\mu$ in $G$, then similarly we also have
\begin{equation*}
    M_{\bbQ_p}=\prod_{i}M_i
\end{equation*}
where $M_i$ is the centralizer of $\mu_i$ in $M$. Thus, we also have an isomorphism of Hecke algebras
\begin{equation}
\label{eqn:m}
    \calH(M(\bbQ_p)//M(\bbZ_p),\bbQ)\simeq \bigotimes_{i}\calH(M_i(\bbQ_p)//M_i(\bbZ_p),\bbQ).
\end{equation}
For a quasi-split reductive group $G$ with standard parabolic subgroup $P$ and Levi subgroup $M$, we can define following algebra homomorphism, known as the twisted Satake homomorphism 
\begin{equation*}
    \dot{\calS}^G_M:\calH(G(\bbQ_p)//K_p,\bbQ)\rightarrow \calH(M(\bbQ_p)//M_(\bbZ_p),\bbQ),
\end{equation*}
defined as follows. Write $P=NM$, for $N$ the unipotent radical of $P$, and given a function $f\in \calH(G(\bbQ_p)//K_p,\bbQ)$, we have
\begin{equation*}
    \dot{\calS}^G_M(f)(m)=\int_{n\in N}f(nm) dn.
\end{equation*}
The twisted Satake homomorphism also factors: we have an isomorphism
\begin{equation*}
    \dot{\calS}^G_M=\bigotimes_i\dot{\calS}^{H_i}_{M_i}.
\end{equation*}
Consider now the representation $\rho_{\mu_i}:\hat{G}\rightarrow \GL(V_{\mu_i})$ of $\hat{G}$ with highest weight cocharacter the dominant Weyl conjugate of $(1,\dots,\mu^{-1}_i,\dots,1)$, where $\mu^{-1}_i$ is in the $i$-th position. Define the polynomial 
\begin{equation*}
    H_i(x)\in \calH(H_i(\bbQ_p)//H_i(\bbZ_p),\bbQ)[x]
\end{equation*}
as the polynomial given by
\begin{equation}
\label{eqn:h1}
    H_i(x)=\det(x-p^{n_i}\rho_{\mu_i}(\sigma\ltimes \hat{g})),
\end{equation}
where $n_i=\langle\rho_i,\mu_i\rangle$, where $\rho_i$ is the half sum of positive roots of $G_i$. Note that since $\mu$ is central in $M$, $\mu^{-1}_i$ is also central in $M_i$, hence we can consider the element $1_{\mu^{-1}_i(p)M_{i}(\bbZ_p)}\in \calH(M_{i}(\bbQ_p)//M_{i}(\bbZ_p),\bbQ)$.
\begin{proposition}
\label{prop:abstractES}
We have the following equality in $\calH(M_{i}(\bbQ_p)//M_{i}(\bbZ_p),\bbQ)$: for all $i$,
\begin{equation*}
    H_i(p^{\langle 2\rho,\mu_i\rangle}1_{\mu^{-1}_i(p)M_{i}(\bbZ_p)})=0.
\end{equation*}
\end{proposition}
\begin{proof}
This follows from the same proof as \cite[Prop 3.4]{B2002}, and the observation that under the decomposition \eqref{eqn:m}, we see that $H_i(x)$ corresponds to the polynomial with coefficients in $\calH(M_i(\bbQ_p)//M_{i}(\bbZ_p),\bbQ)$ defined similarly as in \eqref{eqn:h1}, where we instead take the determinant of highest weight representation of $\hat{G_i}$ corresponding to the dominant Weyl conjugate of $\mu^{-1}_i$. The $p^{\langle 2\rho,\mu_i\rangle}$ factor is because of the dot Weyl action defined in \cite[1.8]{W2000} to pass from the dominant Weyl conjugate of $\mu^{-1}_i$ to $\mu_i^{-1}$.
\end{proof}
\subsection{Proof of Eichler-Shimura relations}
We will now show the last part of Theorem \ref{ref:mainthms2} on the level of algebraic correspondences. 
\begin{proposition}
\label{prop:muordinary}
  Let $(G,X)$ be a Shimura datum of abelian type, such that $G=\mathrm{Res}_{F/\bbQ}H$, and $p$ a prime satisfying the conditions in Proposition \ref{prop:pconditions}. Let $H_i(t)$ be the polynomial defined in \eqref{eqn:h1}, viewed as a polynomial with coefficients in $\mathrm{Cor}(\bar{\calS}_{K_p}(G,X),\bar{\calS}_{K_p}(G,X))$ via $h$. Then we have the equality
  \begin{equation*}
    H_i(\Frob_{\frakp_i})=0.
\end{equation*}
in the ring $\mathrm{Cor}(\bar{\calS}_{K_p}(G,X),\bar{\calS}_{K_p}(G,X))$.
\end{proposition}
\begin{proof}
Firstly, observe that from Proposition \ref{prop:dense}, all the terms appearing in $H_i(\Frob_{\frakp_i})$ have a dense ordinary locus, hence it suffices to apply Proposition \ref{prop:Bueltelvanishing} to show that $H_i(\Frob_{\frakp_i})=0$, if we can show that
\begin{equation*}
    H_i(\Frob_{\frakp_i})\cdot x=0
\end{equation*}
for all $x\in \sS_{K_{p}}(G,X)(\overline{\bbF}_p)^{\mathrm{ord}}$.

Let the Hecke polynomial be
\begin{equation*}
    H_i(t)=\sum_{j}A_jt^j,
\end{equation*}
for elements $A_j\in\calH(G_i(\bbQ_p)//G_i(\bbZ_p))$. Using the map $h$ defined in \ref{eqn:defh}, we can let $h(A_j)$ denote the mod $p$ algebraic cycle in $\mathrm{Cor}(\barSs_{K_p}(G,X),\barSs_{K_p}(G,X))$ corresponding to $A_j$. Thus, we want to show that
\begin{equation}
    \left(\sum_j h(A_j)\cdot \Frob^j_{\frakp_i}\right)\cdot x=0.
    \label{eqn:1}
\end{equation}

The proof then follows as in \cite[Thm 4.7]{B2002}. We write the coefficients of the Hecke
polynomial $A_j$ in terms of left $G_{i}(\bbZ_p)$-cosets
\begin{equation*}
    A_j = \sum_k n^{(j)}_km^{(j)}_kG_{i}(\bbZ_p),
\end{equation*}
where we further apply the Iwasawa decomposition so that $n^{(j)}$ lies in $U_i(\mathbb{Q}_p)$, and $m^{(j)}$ lies in $M_i(\mathbb{Q}_p)$. Here $U_i$ is the unipotent radical of the standard parabolic subgroup $P_i$ corresponding to $\mu_i$. 

If we let $\tilde{x}'$ be the special point lift of $x$ as constructed in \ref{subsub:speciallift}, then we may apply Proposition \ref{prop:comparemodp} relating the actions on the generic fiber and the special fiber to rewrite the LHS of \eqref{eqn:1} as follows. Firstly, Proposition \ref{prop:Frob_i} allows us to replace $\Frob_i^j\cdot x$ with $p^{\langle 2\rho,\mu_i\rangle}[\Theta_x(\mu_i^{-1}(p^j))]$, while Proposition \ref{prop:frob_i} shows us that the mod $p$ reduction of $\mu^{-1}_i(p^{j})\tilde{x}'$ is the same as $\Theta_x(\mu_i^{-1}(p^j))$. Moreover, we can apply Lemma \ref{lemma:genericHecke} to write the action of the correspondences $h(A_j)$ in characteristic zero in terms of cosets. In particular, we see that the LHS of \eqref{eqn:1} is the same as the mod $p$ reduction of
\begin{equation}
\sum_{j,k} p^{j\langle 2\rho,\mu_i\rangle}\pi_{K_p}(n^{(j)}_k m_k^{(j)}\mu^{-1}_i(p^{j})\tilde{x}').
\end{equation}
Moreover, we see from Proposition \ref{prop1} that we have an equality of mod $p$ reductions
\[\sum_{j,k} p^{j\langle 2\rho,\mu_i\rangle}\left[\overline{\pi_{K_p}(n^{(j)}_k m_k^{(j)}\mu^{-1}_i(p^{j})\tilde{x}})\right]=\sum_{j,k}p^{j\langle 2\rho,\mu_i\rangle}\left[\overline{\pi_{K_p}(m_k^{(j)}\mu^{-1}_i(p^{j})\tilde{x}})\right],\]
since the $n_k^{(j)}\in U_i(\bbQ_p)$ do not change the mod $p$ reduction. Now, the abstract Eichler-Shimura relation in Proposition \ref{prop:abstractES} implies since $\dot{\calS}^{H_i}_{M_i}(A_j) = \sum_k m^{(j)}_kM_{i}(\bbZ_p)$, we may write the polynomial $H_i(p^{\langle 2\rho,\mu_i\rangle}1_{\mu_i^{-1}(p)M_c})$ as a sum of left $M_i(\bbZ_p)$-cosets,
\[\sum_{j,k}p^{j\langle 2\rho,\mu_i\rangle}m_k^{(j)}\mu^{-1}_i(p^{j})M_i(\bbZ_p),\]
and since this sum is zero, together with the fact that $M_i(\bbZ_p)$ acts trivially on $\tilde{x}'$, we can conclude.

\end{proof}
\subsection{Passage to intersection cohomology}
We now explain here how to get Theorem \ref{ref:mainthms2} on the action on intersection cohomology, from the previous result on algebraic correspondences. This is well known, but we include it here for the reader's convenience. For this, we are using the formalism of cohomological correspondence, explained for instance in \cite[A.2]{XZ17}.
\subsubsection{} In the following, all functors will be derived functors for notational simplicity. We first note that as explained in \cite[\S7.1]{Far2004}, for each Hecke correspondence $C_g$ on the generic fiber, we have an associated isomorphism of complexes $u_g:p_2^*\calL_{\xi,\ell}\xrightarrow{\sim} p_1^*\calL_{\xi,\ell}\simeq p_1^!\calL_{\xi,\ell}$ supported on $C_g$. (Note that the correspondence acts from right to left).

Taking nearby cycles functor $\Psi$, as constructed in \cite[6.3.2]{Far2004} since $p_1,p_2$ are proper we get a map 
\[\Psi(u_g):p_2^*\Psi\calL_{\xi,\ell}\rightarrow \Psi p_2^*\calL_{\xi,\ell}\xrightarrow{\Psi(u_g)} \Psi p_1^!\calL_{\xi,\ell}\rightarrow p_1^!\Psi\calL_{\xi,\ell}\]
of complexes of sheaves on $C_{g,0}$.

In the proof of \ref{thm:Lanstroh}, Lan-Stroh show that since $\barSs_K(G,X)$ is smooth, $\Psi\calL_{\xi,\ell}$ is isomorphic to $\calL_{\xi,\ell}$, where here we abuse notation to denote by $\calL_{\xi,\ell}$ the special fiber of canonical extension of $\calL_{\xi,\ell}$ to $\sS_K(G,X)$. Thus, this gives us an associated map of complexes $\Psi(u_g):p_2^*\calL_{\xi,\ell}\rightarrow p_1^!\calL_{\xi,\ell}$. 

Let $C_{g,0}^{\min}$ denote the Zariski closure of $C_{g,0}$ in $\barSs_K(G,X)^{\min}\times \barSs_K(G,X)^{\min}$, and we let $j_g: C_{g,0}\hookrightarrow C_{g,0}^{\min}$, $j:\barSs_K(G,X)\hookrightarrow\barSs_K(G,X)^{\min}$ the respective open embeddings. Note that, as explained in \cite[\S3]{Wu2022}, using Morel's work on $t$-structures \cite{Morel2006}, since $\calL_{\xi,\ell}$ is pure of some weight $m_\xi$, 
we can extend $\Psi(u_g)$ to a map
\[j_{g!*}\Psi(u_g):p_2^*\IC(\calL_{\xi,\ell})\rightarrow p_1^!\IC(\calL_{\xi,\ell})\]
which from \cite[Lemma 3.4]{Wu2022} is moreover compatible with composition of cohomological correspondences.

We now define the cohomological correspondence on $\Frob_i$. As explained in \cite[\S3]{LanStroh2}, the $\ell$-adic local system $\calL_{\xi,\ell}$ is induced by the $G(\bbQ_\ell)$-action on the tower of Shimura varieties. More precisely, if we let $V_\xi$ be the $\bar{\bbQ}_\ell$-vector space of the representation $\xi$, we can look at a $\bar\bbZ_\ell$-lattice $\Lambda$ of $V_\xi$, and for each $m$ we can look at $\Lambda_{0,m}=\Lambda_0/\ell^m\Lambda_0$. Let $U_m=\ker(G(\bbZ_\ell)\rightarrow G(\bbZ_\ell/\ell^m\bbZ_\ell))$ and define the torsion \'{e}tale sheaf $V_{\ell^m}$ of sections of the contracted product
\[\sS_K(G,X)\times^{G(\bbQ_\ell)/U_m} \underline{\Lambda_{0,m}},\]
where $\underline{\Lambda_{0,m}}$ is the associated constant group scheme over $O_{E_v}$. we then have
\[\calL_{\xi,\ell}\simeq \lim_{m}V_{\ell^m}.\]
In particular, since $\ell\neq p$, we see that over $\Frob^\circ_i$ we have an isomorphism $p_2^*\calL_{\xi,\ell}\xrightarrow{\sim} p_1^*\calL_{\xi,\ell}$. It remains to observe that the cycle class of $[\Frob_i]$ induces, via the identification with Borel-Moore homology explained in \cite[A.2.2]{XZ17}, a map $\bar{\bbQ}_\ell\rightarrow p_1^!\bbQ_\ell$, which after tensoring with the isomorphism above gives us a cohomological correspondence
\[u_{\Frob_i}:p_2^*\calL_{\xi,\ell}\rightarrow p_1^!\calL_{\xi,\ell},\]
supported on $\Frob_i^\circ$ which we can also take intermediate extension to $\Frob_i^{\min}$, the Zariski closure of $\Frob_i$ in $\barSs_K(G,X)^{\min}\times \barSs_K(G,X)^{\min}$ to get a cohomological correspondence
\[j_{\Frob!*}u_{\Frob_i}:p_2^*\IC(\calL_{\xi,\ell})\rightarrow p_1^!\IC(\calL_{\xi,\ell}).\]
Note here that taking intermediate extensions from $\barSs_K(G,X)^\ord$ to $\barSs_K(G,X)^{\min}$ is the same as taking intermediate extensions from $\barSs_K(G,X)$ to $\barSs_K(G,X)^{\min}$, since both $\barSs_K(G,X)^\ord$ and $\barSs_K(G,X))$ are smooth.

Finally, to conclude cohomological vanishing, let $j_{!*}u_{H_i(\Frob_i)}$ denote the cohomological correspondence given by taking sums, compositions of $j_{\Frob!*}u_{\Frob_i}$ and $j_{g!*}\Psi(u_g)$ as in the expression for $H_i(\Frob_i)$. Observe that since as algebraic cycles on $\barSs_K(G,X)^{\min}\times \barSs_K(G,X)^{\min}$ we have $H_i(\Frob_i)=0$, the underlying support of the cohomological correspondence vanishes, and thus $j_{!*}u_{H_i(\Frob_i)}$ must also vanish.
\section{Semisimplicity Criterion}
\subsection{An abstract semisimplicity criterion}
We first recall the following theorem of Fayad and Nekov\'{a}\v{r} \cite[Theorem 1.7]{FN19}. Let $\Gamma_E=\mathrm{Gal}_E$.
\begin{definition}
A continuous representation $\rho:\Gamma_E\rightarrow \GL(V)$ is said to be \emph{strongly irreducible} if for any open finite index subgroup $U\subset \Gamma_E$, $\rho|_{U}$ is irreducible.
\end{definition}
\begin{theorem}
\label{thm:criterion}
Let $\Gamma$ be a profinite group, $V,W_1,...,W_r$ non-zero vector spaces of finite dimension over $\bar{\bbQ}_\ell$. Let $\rho:\Gamma\rightarrow \GL(V)$ and $\rho_i:\Gamma\rightarrow \GL(W)$ be continuous representations of $\Gamma$ with  Lie algebras 
\begin{equation*}
    \frakg_i= \Lie(\rho_i(\Gamma)),\qquad \frakg= \Lie(\rho(\Gamma)).
\end{equation*}
We denote $\bar{\frakg}_i=\frakg_i\otimes\bar{\bbQ}_\ell,\bar{\frakg}=\frakg\otimes\bar{\bbQ}_\ell$. If the following three conditions hold, then the representation $\rho=\rho^{ss}$ is semisimple.
\begin{enumerate}
    \item Each $\rho_i$ is strongly irreducible (which implies that each $\frakg_i$ is a reductive $\bar{\bbQ}_\ell$-Lie algebra and each element of its centre acts on $W_i$ by a scalar).
    \item For each $i= 1,\dots,r$, every (equivalently, some) Cartan subalgebra $\frakh_i$ of $\frakg_i$ acts on $W_i$ without multiplicities (i.e., all weight spaces of $\frakh_i$ on $ W_i$ are one-dimensional).
    \item There exists an open subgroup $\Gamma'\subset \Gamma$ and a dense subset $\Sigma\subset \Gamma'$ such that for each $g\in\Sigma$ there exists a finite dimensional vector space over $\bar{\bbQ}_\ell$ (depending on $g$) $V(g)\supset V$ and elements $u_1,\dots,u_r\in \GL(V(g))$ such that $u_iu_j=u_ju_i$, $P_{\rho_i(g)}(u_i) = 0$ for all $i$, $j= 1,...,r$, and $V$ is stable under $u_1\dots u_r$ and $u_1\dots u_r|_{V}=\rho(g)$
\end{enumerate}
\end{theorem}
Here, $P_{\rho_i(g)}$ is the characteristic polynomial of $\rho_i(g)$.

We state here a theorem of Sen \cite[Theorem 1]{Sen1973} which we will use to find representations which satisfy condition (2) of the theorem above, given a Galois representation $\rho$.
\begin{proposition}
\label{prop:sen}
Let $\frakg=\bar{\bbQ}_l\cdot\mathrm{Lie}(\rho(\Gamma_E))\subset \mathfrak{gl}(n,\bar{\bbQ}_\ell)$ be the $\bar{\bbQ}_\ell$-Lie algebra generated by the image of $\rho$. If $\rho|_{G_{E_\tau}}$ is Hodge-Tate for any $\tau:E\hookrightarrow \bar{\bbQ}_\ell$, then any Cartan subalgebra $\frakh\subset\frakg$ acts on $\bar{\bbQ}_\ell^n$ by the $n$ Hodge-Tate weights of $\rho|_{G_{E_\tau}}$.
\end{proposition}
\subsection{Proof of Main Theorem}
We prove in this subsection Theorem \ref{thm:main}. Let us recall the setup: we have a representation $\sigma_{\pi^\infty}:\Gamma_E\rightarrow V^i(\pi^\infty)$ appearing in the intersection cohomology of some Shimura variety $(G,X)$, where $G=\mathrm{Res}_{F/\mathbb{Q}}H$, associated with some automorphic representation $\pi$ of $H(\mathbb{A}_F,f)$. We also assume that we have attached via the Langlands correspondence $\rho_{\pi,\mu_i}:\Gamma_E\rightarrow \phantom{}^{L}G\xrightarrow{r_{-\mu_i}}\GL(V_{-\mu_i})$. We restate Theorem \ref{thm:main} here:

\begin{theorem}
Let $(G,X)$ be a Shimura datum of abelian type such that $G=\mathrm{Res}_{F/\bbQ}H$ for some connected reductive group $H$, and totally real number field $F$. Let $\pi$ be an automorphic representation of $G(\bbA_{\bbQ,f})=H(\bbA_{F,f})$. For all $v$, suppose that the $^{L}H$-valued Galois representation associated to $\pi$ exists, and we consider for all $v|\infty$ the composition with the highest weight representation $\rho_{\pi,\mu_v}:\mathrm{Gal}(\bar{F}/F)\rightarrow {^L}H\rightarrow \GL(V_{-\mu_v})$. Suppose that moreover we also know that 
\begin{enumerate}
    \item $\rho_{\pi,\mu_v}$ is strongly irreducible
    \item For all primes $\frakp$ of $E$ such that $\frakp|l$, the Hodge-Tate weights of $\rho_{\pi,\mu_v}|_{D_\frakp}$ are distinct.
\end{enumerate}
Then $\sigma_{\pi^\infty}$ is a semisimple representation.
\end{theorem}
\begin{proof}
In the notation of Theorem \ref{thm:criterion}, we have $\Gamma=\Gamma_E=\mathrm{Gal}(\bar{E}/E)$, $\rho=\sigma_{\pi^\infty}$, $V=V^i(\pi^\infty)$, and each $\rho_i$ is $\rho_{\pi,\mu_i}$, $W_i=V_{-\mu_i}$. We will apply this criterion where the elements $g$ are the Frobenius elements at the primes $\Frob_p$ which are split, and the $u_i$ will be the partial Frobenius $\Frob_{\frakp_i}$. Moreover, to show that the representation $\rho$ is semisimple, it suffices to show that $\rho|_{U}$ is semisimple for any open finite index subgroup of $\mathrm{Gal}(\bar{E}/E)$, and thus, we may restrict to primes $p$ where $G_{\bbQ_p}$ is split, since such primes have a positive density.

From Theorem \ref{ref:mainthms2}, observe that 
\begin{equation*}
    H_{\mu_i}(\Frob_{\frakp_i}|V^i(\pi^\infty))=0.
\end{equation*}
Note that each summand $V^i(\pi^\infty)\otimes (\pi^\infty)^K$ in $IH^i_{\acute{e}t}(\Sh_K(G,X)^{\min}_{\bar{E}},\IC(\mathcal{L}_{\xi,\ell}))$ is stable for the action of $\Frob_{\frakp_i}$, since the action of $G(\mathbb{A}_f^p)$ and $\mathcal{H}_i$-actions commute with the $\Frob_{\frakp_i}$-action.

Replacing each element of the Hecke algebra $\mathcal{H}_i$ with its eigenvalue on $\pi_p^{K_p}$, we obtain
\begin{equation*}
    H_{\mu_i}|_{\pi_p^{K_p}}(\Frob_{\frakp_i}|_{V^i(\pi^\infty)\otimes (\pi^\infty)^K})=0.
\end{equation*}
Observe that the polynomial on the right hand side is, by the definition of the Hecke polynomial $H_{\mu_i}$, the characteristic polynomial $\mathrm{det}(t-\rho_i(\Frob_p))=:P_{\rho_i(g)}$.

In order to show semisimplicity of the endomorphism given by $\Frob_p$, it suffices to show that the $P_{\rho_i}$ has distinct roots. This is an open condition on $\mathrm{Gal}(\bar{E}/E)$, hence it suffices to exhibit an element $u\in \mathrm{Gal}(\bar{E}/E)$ such that the characteristic polynomial of $\rho_i(u)$ has distinct roots. To see this, observe that Proposition \ref{prop:sen} applies, hence the Lie algebra
\begin{equation*}
    \bar{\frakg}=\bar{\bbQ_l}\cdot \mathrm{Lie}(\rho_i(\Gamma_E))
\end{equation*}
contains a semisimple element whose eigenvalues on $V_i(\pi^\infty)$ act by the Hodge-Tate weights of $\rho_i|_{D_v}$, where $v$ is a prime dividing $\ell$. By assumption (2), the Hodge-Tate weights of $\rho_i|_{D_v}$ are all distinct. This implies that there is an open subset of $\mathrm{Gal}(\bar{E}/E)$ where the characteristic polynomial of $\rho_i(u)$ has distinct roots.

Finally, we conclude using the criterion in Theorem \ref{thm:criterion}, since we have constructed a Zariski dense set of elements which are semisimple, that $\rho$ is semisimple.
\end{proof}
\section{Applications to some Shimura varieties}
We discuss here some examples of Shimura varieties where Theorem \ref{thm:main} applies. 
\subsection{Proof of Theorem \ref{thm:GSpapplication}} Let $F$ be a totally real field and consider the Shimura variety associated to the group $G=\mathrm{Res}_{F/\bbQ}H$, where $H$ is some inner form of $G^*$, defined as one of the following groups:
\begin{description}
\item[symplectic] $G^*=\GSp_{2g}$
\item[orthogonal, $n$ even]  $G^*=\mathrm{GSO}_{2n}$
\item[orthogonal, $n$ odd] $G^*$ is a non-split quasi-split form of $\mathrm{GSO}_{2n}$ relative to $E/F$, a CM extension
\end{description}

We recall the statement of Theorem \ref{thm:GSpapplication} for the reader's convenience:
\begin{theorem}
Let $\pi$ be a cuspidal $L$-algebraic automorphic representation of $H(\bbA_F)$, satisfying
\begin{enumerate}
    \item There is a finite $F$-place $v_{St}$ such that $H_{F_{v_{St}}}$ and $G^*_{F_{v_{St}}}$ are isomorphic, and under this isomorphism $\pi_{v_{St}}$ is the Steinberg representation of $G^*(F_{v_{St}})$ twisted by a character.
    \item $\pi_\infty|\mathrm{sim}|^{-n(n-1)/4}$ is $\xi$-cohomological for an irreducible algebraic representation $\xi = \otimes_{y:F\rightarrow \bbC} \xi_y$ of the group $(\mathrm{Res}_{F/\bbQ}G^*)_\bbC$, where $\mathrm{sim}$ is the similitude factor map $\mathrm{sim}: G^* \rightarrow \bbG_m$.
    \item The representation $\pi_v$ is regular after composing with the representation $\mathrm{GSpin}_{2g+1}\xrightarrow{\mathrm{spin}}\GL_{2^g}$ if symplectic (resp. $\mathrm{GSpin}_{2n} \xrightarrow{\mathrm{std}} \GL_{2n}$ if orthogonal) at every infinite place $v$ of $F$.
\end{enumerate} 
If moreover the $\ell$-adic Galois representation $\phi_\pi:\mathrm{Gal}(\bar{F}/F)\rightarrow \widehat{H}$ satisfies
\begin{enumerate}
\setcounter{enumi}{3}
    \item The image of $\phi_\pi$ is Zariski dense in $\widehat{H}$,
\end{enumerate}
Then the Galois module 
\begin{equation*}
    \mathrm{Hom}_{G(\bbA_f)}(\pi^\infty,IH_{\acute{e}t}^d(\Sh(G,X)^{\min}_{\bar{E}},\IC(\calL_{\xi,\ell})))
\end{equation*}
is semisimple.
\end{theorem}

Firstly, note that under condition (1) and (2) on $\pi$, Kret and Shin \cite[Theorem A]{KS2020galois} and \cite[Theorem A]{KS2020} construct the Galois representation
\begin{description}
    \item[symplectic]$\rho_\pi: \mathrm{Gal}(\bar{F}/F) \rightarrow \mathrm{GSpin}_{2n+1}(\bbQ_\ell)$
    \item[orthogonal]$\rho_\pi: \mathrm{Gal}(\bar{F}/F) \rightarrow \mathrm{GSpin}_{2n}(\bbQ_\ell)$
\end{description}
associated to $\pi$. In fact, from their arguments, we may also extract the following proposition about concentration in middle degree:
\begin{proposition}
    \label{prop:Steinbergmiddledegree}
    Suppose that $\pi$ is a cuspidal $L$-algebraic automorphic representation of $G(\mathbb{A}_{f})$ which satisfies that there is a finite $F$-place $v_{\mathrm{St}}$ such that $\pi_{v_{\mathrm{St}}}$ is a twist of Steinberg. Then we have $IH^i_{\acute{e}t}(\Sh(G,X)^{\min}_{\bar{E}},\IC(\calL_{\xi,\ell}))[\pi]=0$ for all $i\neq d:=\dim \Sh(G,X)$.
\end{proposition}
\begin{proof}
    For the case of symplectic groups, this is explained in the last paragraph of \cite[Proposition 8.2]{KS2020galois}, while for orthogonal groups, this is the last two parargraphs of \cite[Proposition 9.6]{KS2020}.
\end{proof}
\begin{remark}
    This result relies on Arthur's results in \cite{Art13}, and is conditional on the twisted weighted fundamental lemma in general.
\end{remark}
Now, we explain how to use Theorem \ref{thm:main} to deduce Theorem \ref{thm:GSpapplication}. As explained above, the first two conditions are necessary for the construction of the corresponding Galois representation. Now, we see that if we have condition (4), then at all places $v|\infty$ where the group is not compact modulo center, the associated representation $\rho_{\pi,\mu_v}$ will be strongly irreducible. This is because $\rho_{\pi,\mu_v}$ is irreducible, since it does not factor through any Levi subgroup, and if we look at the Zariski closure of the image in $\phi_\pi$ of any finite index open subgroup, it must also be $\widehat{H}$, since $\widehat{H}$ is connected, and thus also irreducible.

We remark here that such a condition should hold for `most' representations, for example, this is the generalization of the non-CM condition of \cite{Nek2018}.

Finally, we see that condition (3) implies that the Hodge-Tate weights of $\rho_{\pi,\mu_v}$ are distinct. To see this, we first recall the definition of regularity. Let $\varphi : W_{\mathbb{R}} \rightarrow \GSpin_{m}(\mathbb{C})$ be the Langlands parameter associated with $\pi_v$, then the composition
\[\mathbb{C}^\times \subset W_{\mathbb{R}} \rightarrow \mathrm{GSpin}_{2n+1}(\bar{\bbQ}_\ell)\xrightarrow{\mathrm{spin}}\GL_{2^n} (\bar{\bbQ}_\ell)\]
(symplectic case) or
\[\mathbb{C}^\times \subset W_{\mathbb{R}} \rightarrow \mathrm{GSpin}_{2n}(\bar{\bbQ}_\ell)\xrightarrow{\mathrm{std}}\GL_{2n} (\bar{\bbQ}_\ell)\]
(orthogonal case) is conjugate to the cocharacter $z\mapsto \mu_1(z)\mu_2(\bar{z})$ given by some $\mu_1,\mu_2\in X_*(\hat{T})\otimes_\bbZ \bbC=X^{*}(T)\otimes_\bbZ \bbC$ such that $\mu_1 - \mu_2 \in X^*(T)$. Then we say $\varphi$ is regular if $\mu_1$ is regular as a character of $\GL_m$, that is, $\langle \alpha^\vee,\mu_1\rangle \neq 0$ for all coroots $\alpha^\vee$. In particular, we see that the Hodge-Tate cocharacter (as defined in \cite[\S2.4]{BuzzardGee}) of the composition is $\mu_1$, which is regular, and hence the Hodge-Tate weights are distinct. 
\begin{remark}
    We may also impose a condition in terms of $\xi$, the representation of $G_{\bar{\bbQ}_\ell}$ such that $\pi$ appears in the cohomology of $\calL_{\xi,\ell}$. The Hodge cocharacter $\mu_{\mathrm{Hodge}}$ in this case will be $\lambda_\xi+\rho$, where $\lambda_\xi$ is the dominant highest weight character for the representation of $\xi$, seen as a cocharacter for the dual group. Under the recipe for the Hodge-Tate cocharacter $\mu_{\mathrm{HT}}$ from Theorem A in both \cite{KS2020},\cite{KS2020galois}, we want the composition of $\mu_{\mathrm{HT}}$ with $\mathrm{spin}$ or $\mathrm{std}$ respectively to be regular. 
\end{remark}

We thus satisfy all the conditions to apply Theorem \ref{thm:main}, and we can deduce that the Galois module 
\begin{equation*}
    \mathrm{Hom}_{G(\bbA_f)}(\pi^\infty,IH_{et}^i(\Sh(G,X)_{\bar{E}},\IC(\mathcal{L}_{\xi,\ell})))
\end{equation*}
(which is finite-dimensional over $\bar{\bbQ}_\ell$) is semisimple for any $i$.

We may apply Proposition \ref{prop:Steinbergmiddledegree} to restrict just to middle degree where $i=d$.

\bibliographystyle{amsalpha}
\bibliography{ref.bib}

\end{document}